\documentclass{amsart}

\usepackage{amssymb, enumitem}
\usepackage{mathrsfs}
\usepackage[all]{xy}
\usepackage{hyperref, aliascnt}
\def\today{\number\day\space\ifcase\month\or   January\or February\or
   March\or April\or May\or June\or   July\or August\or September\or
   October\or November\or December\fi\   \number\year}

\newtheorem{lma}{Lemma}[section]

\newaliascnt{thmCt}{lma}
\newtheorem{thm}[thmCt]{Theorem}
\aliascntresetthe{thmCt}

\newaliascnt{corCt}{lma}
\newtheorem{cor}[corCt]{Corollary}
\aliascntresetthe{corCt}

\newaliascnt{propCt}{lma}
\newtheorem{prop}[propCt]{Proposition}
\aliascntresetthe{propCt}

\newtheorem*{thm*}{Theorem}
\newtheorem*{cor*}{Corollary}
\newtheorem*{prop*}{Proposition}

\theoremstyle{definition}

\newaliascnt{pgrCt}{lma}

\aliascntresetthe{pgrCt}

\newaliascnt{dfCt}{lma}
\newtheorem{df}[dfCt]{Definition}
\aliascntresetthe{dfCt}

\newaliascnt{rmkCt}{lma}
\newtheorem{rmk}[rmkCt]{Remark}
\aliascntresetthe{rmkCt}

\newaliascnt{rmksCt}{lma}
\newtheorem{rmks}[rmksCt]{Remarks}
\aliascntresetthe{rmksCt}

\newaliascnt{qstCt}{lma}
\newtheorem{qst}[qstCt]{Question}
\aliascntresetthe{qstCt}

\newaliascnt{pbmCt}{lma}
\newtheorem{pbm}[pbmCt]{Problem}
\aliascntresetthe{pbmCt}

\newaliascnt{notaCt}{lma}
\newtheorem{nota}[notaCt]{Notation}
\aliascntresetthe{notaCt}

\newcommand{\Z}{{\mathbb{Z}}}
\newcommand{\R}{{\mathbb{R}}}
\newcommand{\C}{{\mathbb{C}}}
\newcommand{\N}{{\mathbb{N}}}

\newcommand{\K}{{\mathcal{K}}}
\newcommand{\B}{{\mathcal{B}}}

\newcommand{\Isom}{{\mathrm{Isom}}}
\newcommand{\Full}{{\mathrm{Full}}}
\newcommand{\id}{{\mathrm{id}}}
\newcommand{\supp}{{\mathrm{supp}}}
\newcommand{\Aut}{{\mathrm{Aut}}}
\newcommand{\Max}{{\mathrm{Max}}}
\DeclareMathOperator{\UPF}{UPF}
\DeclareMathOperator{\Ind}{Ind}
\DeclareMathOperator{\ev}{ev}

\newcommand{\SQL}{S\!Q\!L}
\newcommand{\QSL}{Q\!S\!L}
\newcommand{\SL}{S\!L}
\newcommand{\QL}{Q\!L}

\newcommand{\classL}{L}
\newcommand{\classSL}{\SL}
\newcommand{\classQL}{\QL}
\newcommand{\classQSL}{\QSL}
\newcommand{\classSQL}{\SQL}

\DeclareMathOperator{\PF}{PF}

\DeclareMathOperator{\Rep}{Rep}

\newcommand{\ca}{$C^*$-algebra}

\newcommand{\lcg}{locally compact group}

\newcommand{\subQ}{\mathrm{Q}}
\newcommand{\subS}{\mathrm{S}}
\newcommand{\subQS}{\mathrm{QS}}

\title{Group algebras acting on $L^p$-spaces}

\date{\today}

\author{Eusebio Gardella}
\address{Eusebio Gardella
Department of Mathematics, Deady Hall, University of Oregon
Eugene OR, 97403, USA.}
\email{gardella@uoregon.edu}
\urladdr{http://pages.uoregon.edu/gardella}

\author{Hannes Thiel}
\address{Hannes Thiel
Mathematisches Institut, Fachbereich Mathematik und Informatik der
Universit\"at M\"unster, Einsteinstrasse 62, 48149 M\"unster, Germany.}
\email{hannes.thiel@uni-muenster.de}
\urladdr{http://wwwmath.uni-muenster.de/u/hannes.thiel/}

\thanks{The first named author was partially supported by the D.~K. Harrison Prize from the
University of Oregon. The second named author was partially supported by the Deutsche
Forschungsgemeinschaft (SFB 878).}

\subjclass[2010]{Primary:
22D20, 
43A15, 
43A07. 
Secondary:
43A65, 
46E30, 
47L10.
}

\keywords{Locally compact group, $QSL^p$-space, Banach algebra of $p$-pseudofunctions, amenability}

\begin{document}

\begin{abstract}
For $p\in [1,\infty)$ we study representations of a locally compact group $G$ on
$L^p$-spaces and $QSL^p$-spaces.
The universal completions $F^p(G)$ and $F^p_{\mathrm{QS}}(G)$ of $L^1(G)$ with respect to these
classes of representations (which were first considered by Phillips and Runde, respectively),
can be regarded as analogs of the full group \ca{} of $G$ (which is the case $p=2$).
We study these completions of $L^1(G)$ in relation to the algebra $F^p_\lambda(G)$ of $p$-pseudofunctions.
We prove a characterization of group amenability in terms of certain canonical
maps between these universal Banach algebras.
In particular, $G$ is amenable if and only if $F^p_{\mathrm{QS}}(G)=F^p(G)=F^p_\lambda(G)$.

One of our main results is
that for $1\leq p< q\leq 2$, there is a canonical map $\gamma_{p,q}\colon F^p(G)\to F^q(G)$ which is
contractive and has dense range.
When $G$ is amenable, $\gamma_{p,q}$ is injective, and it is never surjective unless $G$ is finite.
We use the maps $\gamma_{p,q}$ to show that when $G$ is discrete, all (or one) of the universal
completions of $L^1(G)$ are amenable as a Banach algebras if and only if $G$ is amenable.

Finally, we exhibit a family of examples showing that the characterizations of group amenability
mentioned above cannot be extended to $L^p$-operator crossed products of topological spaces.
\end{abstract}

\maketitle

\tableofcontents

\section{Introduction}
\label{sec:Intro}

In this paper, we study representations of a \lcg{} $G$ on certain classes of Banach spaces, as well
as the corresponding completions of the group algebra $L^1(G)$.
More specifically, for each $p\in[1,\infty)$, we consider the class $\classL^p$ of $L^p$-spaces;
the class $\classSL^p$ of closed subspaces of Banach spaces in $\classL^p$;
the class $\classQL^p$ of quotients of Banach spaces in $\classL^p$ by closed subspaces; and
the class $\classQSL^p$ of quotients of Banach spaces in $\classSL^p$ by closed subspaces (\autoref{df:importantClasses}).
We denote the corresponding universal completions of $L^1(G)$ by $F^p(G), F^p_\subS(G),
F^p_\subQ(G)$ and $F^p_\subQS(G)$, respectively; see \autoref{df:abbreviations}.

We also study the algebra $F^p_\lambda(G)$ of $p$-pseudofunctions on $G$.
This is the Banach subalgebra of $\B(L^p(G))$ generated by all the operators of left convolution by functions in $L^1(G)$.
Equivalently, $F^p_\lambda(G)$ is the closure of the image of the left regular representation $\lambda_p\colon L^1(G)\to \B(L^p(G))$.
This algebra was introduced by Herz in \cite{Her73SynthSubgps}, where $F^p_\lambda(G)$ was denoted $PF_p(G)$
(see also \cite{NeuRun09ColumnRowQSL}).

The fact that $L^1(G)$ has a contractive approximate identity implies that there exist canonical
isometric isomorphisms
\[
L^1(G)\cong F^1_\lambda(G)\cong F^1(G) \cong F^1_\subS(G)\cong F^1_\subQ(G)\cong F^1_\subQS(G);
\]
see \autoref{prp:F1} and \autoref{rem:poneReduced}.
However, for $p\neq 1$, the existence of a canonical (not necessarily isometric) isomorphism between any of
the algebras $F^p(G)$, $F^p_\subS(G)$, $F^p_\subQ(G)$ or $F^p_\subQS(G)$, and the algebra $F^p_\lambda(G)$,
is equivalent to amenability of $G$; see \autoref{thm:AmenTFAE}.
For $p=2$, the algebra $F^2(G)$ is the full group \ca{} of $G$, usually denoted $C^*(G)$.

Using an extension theorem of Hardin, we show that for $p\notin\{4,6,8,\ldots\}$ and
$q\notin\{\tfrac{4}{3},\tfrac{6}{5},\tfrac{8}{7},\ldots\}$, and regardless of $G$, there are canonical
isometric isomorphisms
\[
F^p_\subS(G) \cong F^p(G) \ \ \mbox{ and } \ \ F^q_\subQ(G) \cong F^q(G).
\]

A consequence of the existence of such isomorphisms is that, for $1\leq p\leq q\leq 2$ and for $2\leq r\leq s<\infty$, there are canonical, contractive homomorphisms
\[
\gamma_{p,q}\colon F^p(G)\to F^q(G) \ \ \mbox{ and } \ \ \gamma_{s,r}\colon F^s(G)\to F^r(G)
\]
with dense range; see \autoref{thm:CanMaps}.
This can be interpreted as saying that the algebras $F^p(G)$ form a `continuous interpolating family' of Banach algebras between the group algebra $F^1(G)=L^1(G)$ and the full group \ca{} $F^2(G)=C^*(G)$.

When $G$ is amenable, our results recover, using different methods, a result announced by Herz as Theorem~C in \cite{Her71pSpApplConv}, and whose proof appears in the corollary on page~512 of \cite{HerRiv72EstimatesMixNorm}.
Furthermore, in this case, we show that for $1\leq p<q\leq 2$ or $2\leq q<p<\infty$, the map $\gamma_{p,q}\colon F^p(G)\to F^q(G)$ is injective, and that it is never surjective unless $G$ is finite;
see \autoref{cor:FpG}.

Another application of \autoref{thm:CanMaps} is as follows: when $G$ is discrete, amenability
of any of the Banach algebras $F^p(G)$, $F^p_\subS(G)$, $F^p_\subQ(G)$ or $F^p_\subQS(G)$, is equivalent to amenability of $G$;
see \autoref{thm:FpGamenableIFF}.
The cases $p=1$ and $p=2$ of this theorem are well-known, the first one being due to B.~Johnson \cite{Joh72CohomologyBAlg}, and holding even if $G$ is not discrete.

A partial summary of our results on characterization of group amenability is as follows.
(The equivalence between (1) and (2) below, specifically for $L^p$-spaces, was independently
obtained by Phillips, whose methods are inspired in $C^*$-algebraic techniques.)

\begin{thm*}
Let $G$ be a locally compact group and let $p\in (1,\infty)$.
Consider the following statements:
\begin{enumerate}
\item
The group $G$ is amenable.
\item
The canonical map from each of the algebras $F^p(G)$, $F^p_\subS(G)$, $F^p_\subQ(G)$, or $F^p_\subQS(G)$, to $F^p_\lambda(G)$, is an isometric isomorphism.
\item
The canonical map from any of the algebras $F^p(G)$, $F^p_\subS(G)$, $F^p_\subQ(G)$, or $F^p_\subQS(G)$, to $F^p_\lambda(G)$, is a (not necessarily isometric) isomorphism.
\item
The algebras $F^p(G),F^p_\subS(G),F^p_\subQ(G)$, and $F^p_\subQS(G)$, are amenable.
\item
At least one of $F^p(G),F^p_\subS(G),F^p_\subQ(G)$, or $F^p_\subQS(G)$, is amenable.
\end{enumerate}
Then (1) $\Leftrightarrow$ (2) $\Leftrightarrow$ (3) $\Rightarrow$ (4) $\Rightarrow$ (5). If $G$ is discrete,
then all five statements are equivalent.
\end{thm*}

Finally, we show in \autoref{thm:ComputationCP} that the theorem above, particularly the implications `(2) $\Rightarrow$ (1)' and `(4) $\Rightarrow$ (1)', cannot be generalized to $L^p$-crossed products of discrete groups acting on algebras of the form $C_0(X)$, for some locally compact Hausdorff space~$X$.

Further connections between $G$ and $F^p_\lambda(G)$ will be explored in \cite{GarThi14pre:IsoConv};
functoriality properties of $F^p_\lambda(G)$ (and those of the universal completions of $L^1(G)$ discussed
above) are studied in \cite{GarThi14arX:Functoriality}; and applications of the results in this paper
are given in \cite{GarThi14arX:LpGenInvIsom}.

We have made the effort of not adopting any cardinality assumptions ($\sigma$-finiteness of measures, second-countability
or $\sigma$-compactness of groups, or separability of Banach spaces) whenever possible. This implies considerable
additional work when showing the existence of the maps $\gamma_{p,q}\colon F^p(G)\to F^q(G)$; see \autoref{thm:CanMaps}.
Indeed, some of the techniques for dealing with $L^p$-spaces require the involved measure spaces
to be $\sigma$-finite (as is the case with Hardin's theorem).
In order to directly apply such techniques, one has to restrict to second-countable (or sometimes $\sigma$-compact) \lcg{s}.
However, one can often reduce a problem about a \lcg{} $G$ to second-countable \lcg{s} by applying the following two steps:
First, $G$ is the union of open, $\sigma$-compact subgroups.
Second, by the Kakutani-Kodaira Theorem, every $\sigma$-compact, \lcg{} $H$ contains an arbitrarily small compact, normal subgroup $N$ such that $H/N$ is second-countable.
This technique is for instance used to prove \autoref{thm:NormsLcg}.

\vspace{0.3cm}

All locally compact groups will be endowed with their left Haar measure.
We take $\N=\{1,2,\ldots\}$.
For $n$ in $\N$ and $p\in [1,\infty]$, we write $\ell^p_n$ in place of $\ell^p(\{1,\ldots,n\})$, and we write $\ell^p$ in place of $\ell^p(\Z)$.

Let $E$ be a Banach space.
We write $E_1$ for the unit ball of $E$ and $E'$ for its dual space.
If $F$ is another Banach space, we denote by $\B(E,F)$ the Banach space of all bounded linear operators $E\to F$, and write $\B(E)$ in place of $\B(E,E)$.
We denote by $\Isom(E)$ the subset of $\B(E)$ consisting of invertible isometries.
(In this paper, invertible isometries will always be assumed to be surjective.)
The group $\Isom(E)$ will be endowed with the strong operator topology.
It is a standard fact that $\Isom(E)$ is a Polish group whenever $E$ is a separable Banach space.
Moreover, it is easy to check that if $X$ is a topological (measurable) space, a function
$u\colon X\to\Isom(E)$ is continuous (measurable) if and only if for every $\xi\in E$, the map
$X\to E$ given by $x\mapsto u(x)\xi$ for $x\in X$, is continuous (measurable).

For a bounded linear map $\varphi\colon E\to F$, we will write $\varphi'\colon F'\to E'$ for its dual map.
For $p\in (1,\infty)$, we denote by $p'$ its conjugate (H\"older) exponent, which satisfies $\tfrac{1}{p}+\tfrac{1}{p'}=1$.
\vspace{5pt}

\subsection*{Acknowledgements}
Part of this work was completed while the authors were attending the Thematic Program on Abstract Harmonic Analysis, Banach and Operator Algebras at the Fields Institute in January-June 2014, and while the second named author was visiting the University of Oregon in July and August 2014. The hospitality of the Fields Institute and University of Oregon are gratefully acknowledged.

The authors would like to thank Nico Spronk for helpful conversations,
and Antoine Derighetti and Bill Johnson for electronic correspondence. We are especially
indebted to Chris Phillips for many insightful conversations, as well as for sharing
some of his unpublished work with us, particularly concerning incompressibility
(\cite{Phi14pre:LpIncompr}) and group $\classL^p$-operator algebras (\cite{Phi14pre:LpMultDom}).

\section{Universal completions of \texorpdfstring{$L^1(G)$}{L1(G)}}

Let $G$ be a \lcg.
We let $L^1(G)$ denote the Banach algebra of complex-valued functions on $G$ that are integrable (with respect to the left Haar measure), with product given by convolution.

A representation of a Banach algebra $A$ on a Banach space $E$ is a contractive homomorphism $A\to\B(E)$.

\begin{df}
\label{df:universalCompletion}
Let $\mathcal{E}$ be a class of Banach spaces.
We denote by $\Rep_\mathcal{E}(G)$ the class of all
non-degenerate representations of $L^1(G)$ on Banach spaces in $\mathcal{E}$.
Given $f$ in $L^1(G)$, set
\[
\|f\|_{\mathcal{E}}
= \sup \left\{ \| \pi(f) \| \colon \pi\in\Rep_\mathcal{E}(G) \right\}.
\]
(Note that the supremum exists, even if $\mathcal{E}$ is not a set.)
Then $\|\cdot\|_\mathcal{E}$ defines a seminorm on $L^1(G)$.
Set
\[
I_\mathcal{E}
=\left\{ f\in L^1(G) \colon \|f\|_{\mathcal{E}}=0 \right\}
= \bigcap_{\pi\in \Rep_\mathcal{E}(G)}\ker(\pi),
\]
which is a closed ideal in $L^1(G)$.
We write $F_\mathcal{E}(G)$ for the completion of $L^1(G)/I_\mathcal{E}$ in the induced norm.
\end{df}

\begin{rmk}
Let $G$ be a \lcg, and let $\mathcal{E}$ be a class of Banach spaces.
Then the canonical map $L^1(G)\to F_\mathcal{E}(G)$ is contractive and has dense range.
Moreover, it is injective if and only if the elements of $\Rep_\mathcal{E}(G)$ separate the points
of $L^1(G)$, meaning that for every $f\in L^1(G)$ with $f\neq 0$, there exists $\pi\in\Rep_\mathcal{E}(G)$ such that $\pi(f)\neq 0$.
\end{rmk}

\begin{df}
Let $G$ be a \lcg.
An \emph{isometric representation} of $G$ on a Banach space $E$ is a continuous group homomorphism
of $G$ into the group $\Isom(E)$ of invertible isometries of $E$, where $\Isom(E)$ is equipped with the strong operator topology.
Equivalently, an isometric representation of $G$ is a strongly continuous action of $G$ on $E$ via isometries.
\end{df}

The following result is folklore, and we omit its proof.

\begin{prop}
\label{prop:integratedForm}
Let $G$ be a \lcg\ and let $E$ be a Banach space.
Then there is a natural bijective correspondence between nondegenerate contractive 
representations $L^1(G)\to\B(E)$ and isometric representations of $G$ on $E$.

If $\rho\colon G\to \Isom(E)$ is an isometric representation, then the induced nondegenerate 
contractive representation
$\pi_\rho\colon L^1(G)\to \B(E)$ is given by
\[
\pi_\rho(f)(\xi)=\int_Gf(s)\rho_s(\xi)\ ds
\]
for all $f\in L^1(G)$ and all $\xi\in E$, and it is called the \emph{integrated form} of $\rho$.
\end{prop}

The following are the classes of Banach spaces that we are mostly interested in.

\begin{df}
\label{df:importantClasses}
Let $p\in[1,\infty)$, and let $E$ be a Banach space.
\begin{enumerate}
\item
We say that $E$ is an \emph{$L^p$-space} if there exists a measure space $(X,\mu)$ such that $E$ is isometrically isomorphic to $L^p(X,\mu)$.
\item
We say that $E$ is an \emph{$\SL^p$-space} if there is an $L^p$-space $F$ such that $E$ is isometrically isomorphic to a closed subspace of $F$.
\item
We say that $E$ is a \emph{$\QL^p$-space} if there is an $L^p$-space $F$ such that $E$ is isometrically isomorphic to a quotient of $F$ by a closed subspace.
\item
We say that $E$ is a \emph{$\QSL^p$-space} if there is an $\SL^p$-space $F$ such that $E$ is isometrically isomorphic to a quotient of $F$ by a closed subspace.
\end{enumerate}
We let $\classL^p$ (respectively, $\classSL^p$, $\classQL^p$, $\classQSL^p$) denote the class of all $L^p$-spaces
(respectively, $\SL^p$-spaces, $\QL^p$-spaces, $\QSL^p$-spaces).
\end{df}

\begin{rmks}
Let $p\in[1,\infty)$.

(1)
If $E$ is a separable $L^p$-space, then there exists a finite measure space $(X,\mu)$ such that $E$ is isometrically isomorphic to $L^p(X,\mu)$.
One can also show that every separable $\SL^p$-space is a subspace of a separable $L^p$-space, and analogously for separable $\QL^p$-spaces and $\QSL^p$-spaces. See \cite{GarThi14pre:ReprBAlg}.

(2)
It is a routine exercise to check that if $F$ is a closed subspace of a quotient of an $L^p$-space $E$,
then $F$ is (isometrically isomorphic to) a quotient of a subspace of $E$.
It follows that the class $\classQSL^p$ coincides with the class of Banach spaces that are isometrically isomorphic to a closed subspace of a $\QL^p$-space.
(The latter is the class that one would denote by $\classSQL^p$.)

(3)
Examples of $L^p$-spaces are $L^p([0,1])$ with Lebesgue measure; $\ell^p$ with counting measure;
and $\ell^p_n$ with counting measure, for $n\in\N$.
It is well-known that every separable $L^p$-space is isometrically isomorphic to a countable $p$-direct sum of these.
In particular, up to isometric isomorphism, there are only countably many separable Banach spaces in the class $\classL^p$.
The classes $\classSL^p$ and $\classQL^p$ are much larger, as the following result shows.
\end{rmks}

\begin{prop}
\label{prop:LqSLp}
Let $p,q\in [1,\infty)$.
\begin{enumerate}
\item
If $p \leq q \leq 2$, then $\classSL^p\supseteq\classL^q$.
\item
If $2\leq r \leq s$, then $\classL^r\subseteq\classQL^s$.
\end{enumerate}
\end{prop}
\begin{proof}
Let $1\leq p\leq q\leq 2$.
In Proposition~11.1.9 in \cite{AlbKal06TopicsBSp} it is shown that the space $L^p([0,1])$ isometrically embeds into $L^q([0,1])$.
Essentially the same argument can be used to prove part (1) as follows:
Given an $L^q$-space $E$, it is clear that $E$ is finitely representable in $\ell^q$.
The assumptions on $p$ and $q$ ensure that $\ell^q$ is finitely representable in $\ell^p$.
It follows that $E$ is finitely representable on $\ell^p$.
Therefore, $E$ is isometrically isomorphic to a subspace of some ultrapower of $\ell^p$.
The result follows since $\classL^p$ is closed under ultrapowers.
(See \cite{Hei80UltraprodBSp} and \cite{GarThi14pre:ReprBAlg} for details.)

Part (2) can be deduced from part (1) using duality.
\end{proof}

Since we are mostly interested in the completions of $L^1(G)$ with respect to the classes from \autoref{df:importantClasses}, we will use special notation for these.

\begin{nota}
\label{df:abbreviations}
Let $p\in[1,\infty)$, and let $G$ be a \lcg.
We make the following abbreviations:
\begin{align*}
F^p(G) = F_{\classL^p}(G),\quad
&F^p_\subQ(G) = F_{\classQL^p}(G),\\
F^p_\subS(G) = F_{\classSL^p}(G),\quad
&F^p_\subQS(G) = F_{\classQSL^p}(G).
\end{align*}
\end{nota}

\begin{rmks}
\label{rmk:FpG}
Let $p\in[1,\infty)$, and let $G$ be a \lcg.

(1)
It is well known that every locally compact group $G$ admits a faithful isometric representation on an $L^p$-space (namely, the left regular representation).
It follows that the natural map from $L^1(G)$ to any of $F^p(G), F^p_\subS(G), F^p_\subQ(G)$ and $F^p_\subQS(G)$ is injective.

(2)
Let $\mathcal{E}$ denote any of the classes $\classL^p, \classSL^p, \classQL^p$ or $\classQSL^p$.
Since $\mathcal{E}$ is closed under $p$-direct sums, it follows that that $F_\mathcal{E}(G)$ is an $\mathcal{E}$-operator algebra by the results in \cite{GarThi14pre:ReprBAlg}.
This means that there exists a Banach space $E$ in $\mathcal{E}$ and an isometric representation $F_\mathcal{E}(G)\to\B(E)$.
Moreover, $E$ can be chosen such that the density character of $E$ is dominated by that of $G$; see \cite{GarThi14pre:ReprBAlg}.
In particular, if $G$ is second-countable, then $L^1(G)$ is a separable Banach algebra and consequently $F_\mathcal{E}(G)$ can be isometrically represented on a separable Banach space in $\mathcal{E}$.

(3)
Group representations on $\QSL^p$-spaces have been studied by Volker Runde in \cite{Run05ReprQSL}.
For a locally compact group $G$, he defined the algebra $\UPF_p(G)$ of \emph{universal $p$-pseudofunctions} as follows.
With $\pi\colon L^1(G)\to\B(E)$ denoting a contractive, nondegenerate representation on a $\QSL^p$-space $E$ with the property that any other contractive, nondegenerate representation of $L^1(G)$ on a $\QSL^p$-space is isometrically conjugate to a subrepresentation of $\pi$, the algebra $\UPF_p(G)$ is the closure of $\pi(L^1(G))$ in $\B(E)$ (Definition~6.1 in \cite{Run05ReprQSL}).

It is straightforward to check that there is a canonical isometric isomorphism $\UPF_p(G)\cong F^p_\subQS(G)$.
\end{rmks}

Note that $\classL^2$ is precisely the class of all Hilbert spaces.
Since subspaces and quotients of Hilbert spaces are again Hilbert spaces, we have
$\classQSL^2=\classQL^2=\classSL^2=\classL^2$.
This easily implies the following result.

\begin{lma}
\label{prp:F2}
Let $G$ be a \lcg, and let $f\in L^1(G)$.
Then
\[
\|f\|_{\classQSL^2}=\|f\|_{\classQL^2}=\|f\|_{\classSL^2}=\|f\|_{\classL^2}.
\]
It follows that there are canonical, isometric isomorphisms
\[
F^2_\subQS(G) \cong F^2_\subQ(G) \cong F^2_\subS(G) \cong F^2(G).
\]
\end{lma}

The algebra $F^2(G)$ is in fact a \ca, called the \emph{full group \ca{}} of $G$,
and it is usually denoted $C^*(G)$.

The universal completions from \autoref{notation} also agree when $p=1$, in which case they are all equal
to $L^1(G)$:

\begin{prop}
\label{prp:F1}
Let $G$ be a \lcg, and let $f\in L^1(G)$.
Then
\[
\|f\|_{\classQSL^1}=\|f\|_{\classQL^1}=\|f\|_{\classSL^1}=\|f\|_{\classL^1}=\|f\|_1.
\]
It follows that there are canonical, isometric isomorphisms
\[
F^1_\subQS(G) \cong F^1_\subQ(G) \cong F^1_\subS(G) \cong F^1(G) \cong L^1(G).
\]
\end{prop}
\begin{proof}
Let us denote by $\lambda_1\colon L^1(G)\to \B(L^1(G))$ the integrated form of the left regular representation.
Then $\lambda_1$ is the action of $L^1(G)$ on itself by left convolution.

Given $f$ in $L^1(G)$, it is clear that $\|f\|_{\classL^1}\leq \|f\|_1$.
For the reverse inequality, let $(e_d)_{d\in \Lambda}$ be a contractive approximate identity for $L^1(G)$.
Then
\[
\|\lambda_1(f)\|=\sup\left\{ \tfrac{\|f\ast\xi\|_1}{\|\xi\|_1} \colon \xi\in L^1(G),\xi\neq 0\right\}\geq \sup_{d\in \Lambda} \|f\ast e_d\|_1=\|f\|_1,
\]
so $\|f\|_1\leq \|\lambda_1(f)\|= \|f\|_{\classL^1}\leq \|f\|_1$.
It follows that the norms $\|\cdot\|_1$ and $\|\cdot\|_{\classL^1}$ agree, and thus $F^1(G)=L^1(G)$.

Let $\mathcal{E}$ be any of the classes $\classQSL^1$, $\classSL^1$, or $\classQL^1$.
It follows from the paragraph above that the composition $L^1(G)\to F_\mathcal{E}(G)\to F^1(G)$ equals
the identity map on $L^1(G)$.
The canonical map $F_\mathcal{E}(G)\to F^1(G)$ is therefore isometric, and the result follows.
\end{proof}

We will later show that for most values of $p\in [1,\infty)$, the algebras $F_\subS^p(G)$, $F_\subQ^p(G)$, and $F^p(G)$ are canonically isometrically isomorphic, regardless of $G$; see \autoref{cor:QandSisom}.
When $G$ is amenable, we have $F_\subQS^p(G)=F_\subS^p(G)=F_\subQ^p(G)=F^p(G)$, regardless of $p$; see
\autoref{thm:AmenTFAE}.

\begin{rmk}
\label{rem:poneReduced}
The proof of \autoref{prp:F1} also shows that the algebra of $1$-pseudofunctions $F^1_\lambda(G)$ on $G$ (see \autoref{df:reducedGroupAlgebra}) is canonically isometrically isomorphic to $L^1(G)$ as well.
However, for $p>1$, the analogous result holds if and only if $G$ is amenable; see \autoref{thm:AmenTFAE}.
\end{rmk}

The following observation will allow us to define natural maps between the different universal completions.

\begin{rmk}
\label{rmk:ExistenceCanonicalMaps}
Let $\mathcal{E}_1$ and $\mathcal{E}_2$ be classes of Banach spaces, and suppose that
$\mathcal{E}_1\subseteq \mathcal{E}_2$. Then
\[
\|f\|_{\mathcal{E}_1}
\leq \|f\|_{\mathcal{E}_2}
\]
for all $f\in L^1(G)$, and hence
the identity map on $L^1(G)$ induces a canonical contractive homomorphism $F_{\mathcal{E}_2}(G)
\to F_{\mathcal{E}_1}(G)$.
\end{rmk}

\begin{nota}
\label{nota:kappaSameP}
Let $G$ be a \lcg, and let $p\in[1,\infty)$.
By \autoref{rmk:ExistenceCanonicalMaps}, the inclusions $\classL^p\subseteq\classSL^p\subseteq\classQSL^p$ and $\classL^p\subseteq\classQL^p\subseteq\classQSL^p$ induce canonical contractive homomorphisms between the corresponding universal completions.
We summarize the induced maps in the following commutative diagram:
\[
\xymatrix@=0.90em{
& { F^p_\subS(G) } \ar[dr]^{\kappa^p_\subS} \\
{ F^p_\subQS(G) } \ar[ur]^{\kappa^p_{\mathrm{QS},\mathrm{S}}}
\ar[dr]_{\kappa^p_{\mathrm{QS},\mathrm{Q}}}
& & { F^p(G). } \\
& { F^p_\subQ(G) } \ar[ur]_{\kappa^p_\subQ}
}
\]
We write $\kappa^p_\subQS$ for the composition $\kappa^p_\subS\circ\kappa^p_{\mathrm{QS},\mathrm{S}}$,
which also equals $\kappa^p_\subQ\circ\kappa^p_{\mathrm{QS},\mathrm{Q}}$.
Finally, when the H\"older exponent $p$ is clear from the context, we will drop it from the notation in the natural maps.
\end{nota}

\begin{nota}
\label{nota:kappaPQ}
Let $G$ be a \lcg, and let $p,q\in[1,\infty)$.
If $p\leq q\leq 2$, \autoref{prop:LqSLp} implies that there is a canonical contractive homomorphism
\[
\kappa_{\classSL^p,\classL^q}\colon F^p_\subS(G)\to F^q(G)
\]
with dense range.
Likewise, for $2\leq r\leq s$, there is a canonical contractive homomorphism
\[
\kappa_{\classQL^s,\classL^r}\colon F^s_\subQ(G)\to F^r(G)
\]
with dense range.
\end{nota}

\subsection{Duality}

Recall that if $A$ is a complex algebra, its \emph{opposite algebra}, denoted by $A^{op}$, is the
complex algebra whose underlying vector space structure is the same as for $A$, and the product of two elements $a$ and $b$ in $A^{op}$ is equal to $ba$.
When $A$ is moreover a Banach ($\ast$-)algebra, we take $A^{op}$ to carry the same norm (and involution) as $A$.

An \emph{anti-homomorphism} $\varphi\colon A\to B$ between algebras $A$ and $B$ is a linear map
satisfying $\varphi(ab)=\varphi(b)\varphi(a)$ for all $a$ and $b$ in $A$.
Equivalently, $\varphi$ is a homomorphism $A\to B^{op}$.
Similar terminology and definitions apply to other objects such as topological groups.

\begin{lma}
\label{lma:Gop}
Let $G$ be a \lcg.
Then there is a canonical isometric isomorphism $L^1(G^{op})\cong L^1(G)^{op}$.
Furthermore, if $\mathcal{E}$ is a class of Banach spaces, then there is a canonical isometric isomorphism
$F_\mathcal{E}(G^{op})\cong F_\mathcal{E}(G)^{op}$.
\end{lma}
\begin{proof}
We fix a left Haar measure $\mu$ on $G$, and we denote by $\nu$ the right Haar measure on $G$ given
by $\nu(E)=\mu(E^{-1})$ for every Borel set $E\subseteq G$.
Then $\nu$ is a left Haar measure for $G^{op}$.
It is immediate that inversion on $G$, when regarded as a map $(G,\mu)\to (G^{op},\nu)$, is a measure-preserving group isomorphism.
We denote by $\ast$ the operation of convolution on $L^1(G)$, and by $\ast_{op}$ the operation of convolution on $L^1(G^{op})$ (which is performed with respect to $\nu$).
Similarly, we denote by $\cdot$ the product in $G$, and by $\cdot_{op}$ the product in $G^{op}$.

Given $f$ and $g$ in $L^1(G^{op})$, and given $s\in G^{op}$, we have
\begin{align*}
(f\ast_{op} g)(s)&=\int_G f(t)g(t^{-1}\cdot_{op}s) \ d\nu(t)\\
&=\int_G f(t)g(s\cdot t^{-1}) \ d\nu(t)\\
&=\int_G f(t^{-1})g(s\cdot t) \ d\mu(t)\\
&=(g\ast f)(s).
\end{align*}
It follows that the identity map on $L^1(G)$ induces a canonical isometric anti-isomorphism $L^1(G^{op})\to L^1(G)$, and the result follows.

The last claim is straightforward.
\end{proof}

\begin{rmk}
\label{rem:duality}
We mention, without proof, two easy facts that will be used in the proof of \autoref{prop:DualityUniv}. First, if $\mathcal{E}$ is a class of Banach spaces, and $G$ and $H$ are locally compact groups, then any isomorphism $G\to H$ induces an isometric isomorphism $F_\mathcal{E}(H)\to F_\mathcal{E}(G)$ (the argument is similar to, and simpler than, the one used in the proof of Proposition~2.4 in \cite{GarThi14arX:Functoriality}).
Second, with the same notation, inversion on $G$ defines an isomorphism $G\to G^{op}$, so there are canonical isomorphisms
\[
F_\mathcal{E}(G)\cong F_\mathcal{E}(G^{op})\cong F_\mathcal{E}(G)^{op},
\]
where the second isomorphism is the one given by \autoref{lma:Gop}.
\end{rmk}

\begin{prop}
\label{prop:DualityUniv}
Let $G$ be a \lcg, let $p\in (1,\infty)$, and let $p'$ be its conjugate exponent.
Then there are canonical isometric isomorphisms
\[
F^p_\subQS(G)
\cong F^{p'}_\subQS(G),\quad
F^p_\subQ(G)
\cong F^{p'}_\subS(G),\quad \mbox{ and } \quad
F^p(G)
\cong F^{p'}(G).
\]
\end{prop}
\begin{proof}
Let $\mathcal{E}$ be any class of Banach spaces, and denote by $\mathcal{E}'$ the class
of those Banach spaces that are dual to a Banach space in $\mathcal{E}$.
For $\pi$ in $\mathrm{Rep}_\mathcal{E}(G)$ and $f$ in $L^1(G)$, we have
\[
\|(\pi(f))'\|=\|\pi(f)\|.
\]
Since taking adjoints reverses multiplication of operators, it follows that the identity map on $L^1(G)$ induces a canonical isometric isomorphism $F_\mathcal{E}(G)\cong F_{\mathcal{E}'}(G)^{op}$.
Upon composing this isomorphism with the isomorphism $F_{\mathcal{E}'}(G)\cong F_{\mathcal{E}'}(G)^{op}$
described in \autoref{rem:duality}, we obtain a canonical isometric
isomorphism
\[
F_\mathcal{E}(G)\cong F_{\mathcal{E}'}(G).
\]

To finish the proof, it is enough to observe that for $p\in (1,\infty)$, there are natural identifications
\[
(\classQSL^p)'=\classSQL^{p'}=\classQSL^{p'},\
(\classQL^p)'=\classSL^{p'},\
(\classSL^p)'=\classQL^{p'},\ \mbox{ and } \
(\classL^p)'=\classL^{p'}.
\]
\end{proof}

\subsection{Canonical maps \texorpdfstring{$F^p(G)\to F^q(G)$}{Fp(G)->Fq(G)}}

In this subsection, we will construct, for any \lcg{} $G$, a natural contractive map $F^p(G)\to F^q(G)$ 
with dense range, whenever $1\leq p\leq q\leq 2$ or $2\leq q\leq p<\infty$.

The construction of these maps takes considerable work, since $L^q$-spaces are never $L^p$-spaces, except in trivial cases.
However, it is often the case that an $L^q$-space is a subspace of an $L^p$-space (see \autoref{prop:LqSLp}).
To take advantage of this fact, we need to study extensions of isometries from subspaces of $L^p$-spaces to $L^p$-spaces; see \autoref{thm:Extension}.
Our argument is based on ideas used by Hardin (\cite{Har81IsometriesSubLp}).

The following definition is due to Hardin.

\begin{df}
Let $(X,\mathfrak{A},\mu)$ be a $\sigma$-finite measure space, let $p\in [1,\infty)$ and
let $f_0 \in L^p(X,\mu)$.
Define the \emph{support} of $f_0$ to be
\[
\supp(f_0) = \left\{ x\in X \colon f_0(x)\neq 0 \right\}.
\]
Note that $\supp(f_0)$ is well-defined up to null sets.
If $F$ is a closed subspace of $L^p(X,\mu)$, we say that $f_0$ has \emph{full support in $F$} if
\[
\mu\left(\supp(f)\setminus\supp(f_0)\right)=0
\]
for all $f\in F$.
\end{df}

The following terminology and notations will be convenient.

\begin{df}
Let $F$ be a Banach space.
We denote by $\Isom(F)$ the group of surjective, linear isometries of $F$, and we equip it with the strong operator topology.
In this topology, a net $(u_d)_{d\in \Lambda}$ in $\Isom(F)$ converges to $u\in\Isom(F)$ if and only if $\lim\limits_{d\in \Lambda}\| u_d(\xi) - u(\xi) \|=0$ for every $\xi\in F$.
We call $\Isom(F)$ the \emph{isometry group} of $F$.

If $F$ is a closed subspace of another Banach space $\widetilde{F}$, we let $\Isom(\widetilde{F},F)$ denote the subgroup of $\Isom(\widetilde{F})$ consisting of those isometries that leave $F$ invariant.
\end{df}

The next result asserts that if $p$ is not a multiple of 2 greater than 2, then
for every separable $\SL^p$-space $F$, there exists an $L^p$-space $\widetilde{F}$ containing $F$, such that every invertible isometry on $F$ can be extended to an invertible isometry on $\widetilde{F}$.
Note that this is stronger than the statement that every isometry on a separable $\SL^p$-space can be extended to an isometry on \emph{some} $L^p$-space.

\begin{thm}
\label{thm:Extension}
Let $p\in[1,\infty)\setminus \{4,6,\ldots\}$, and let $F$ be a separable $\SL^p$-space.
Then there exists a separable $L^p$-space $\widetilde{F}$ such that $F$ is isometrically isomorphic to a subspace of $\widetilde{F}$, and such that every surjective, linear isometry on $F$ can be extended to a surjective, linear isometry on $\widetilde{F}$.
Moreover, the restriction map
\[
\varphi\colon \Isom(\widetilde{F},F) \to \Isom(F),
\]
is a surjective homeomorphism.
\end{thm}
\begin{proof}
The statement is trivial for $p=2$, since a closed subspace of a Hilbert space is a Hilbert space itself. We may
therefore assume that $p\in [1,\infty)\setminus\{2,4,6,\ldots\}$.

The proof in this case is based on the proof of Theorem~4.2 in \cite{Har81IsometriesSubLp}, and we use the same notation when possible.
Since $\sigma$-algebras will play an important role, they will not be omitted from the notation.

Let $(X,\mathfrak{A},\mu)$ be a $\sigma$-finite measure space such that $F$ can be identified with a closed subspace of $L^p(X,\mathfrak{A},\mu)$.
We let $\Full(F)$ denote the set of elements in $F$ that have full support in $F$.
It follows from Lemma~3.2 in \cite{Har81IsometriesSubLp} that $\Full(F)$ is nonempty.
Without loss of generality, we may assume that $\supp(f_0)=X$ for every $f_0\in\Full(F)$.

Let $f_0\in\Full(F)$.
Set
\[
Q(f_0) = \left\{ \tfrac{f}{f_0} \colon f\in F \right\},
\]
and let $\sigma(Q(f_0))$ denote the smallest $\sigma$-algebra on $X$ such that every function in $Q(f_0)$ is measurable with respect to $\sigma(Q(f_0))$.
Hardin showed in the proof of Lemma~3.4 in \cite{Har81IsometriesSubLp} that $\sigma(Q(f_0))$ does not depend on the choice of $f_0$ in $\Full(F)$. We will therefore write
$\mathfrak{B}$ for $\sigma(Q(f_0))$, which equals $\sigma(Q(g_0))$ for any other $g_0\in \Full(F)$.
Since every element of $Q(f_0)$ is $\mathfrak{A}$-measurable, we clearly have $\mathfrak{B}\subseteq \mathfrak{A}$.

It is not in general the case that the elements of $Q(f_0)$ are in $L^p(X,\mathfrak{B},\mu)$.
Instead, we shall consider the measure $|f_0|^p\mu$ on $(X,\mathfrak{B})$.
If $f\in F$, then
\[
\|f\|_p
= \left( \int_X |f|^p \ d\mu \right)^{\tfrac{1}{p}} = \left( \int_X \left| \tfrac{f}{f_0} \right|^p  |f_0|^p \ d\mu\right)^{\tfrac{1}{p}}<\infty.
\]
It follows that $Q(f_0)$ is a subset of $L^p\left(X,\mathfrak{B},|f_0|^p\mu\right)$, and that
the map
\[D_{f_0}\colon F\to L^p\left(X,\mathfrak{B},|f_0|^p\mu\right)\]
given by $D_{f_0}(f)=\tfrac{f}{f_0}$ for $f\in F$, is an isometry.
Set
\[
\widetilde{F}
= \left\{ f_0 h \colon h \in L^p\left(X,\mathfrak{B},|f_0|^p\mu\right) \right\}.
\]
By the proof of Theorem~4.2 in \cite{Har81IsometriesSubLp}, the space $\widetilde{F}$ does not depend on the choice of the element $f_0$ of full support.
Moreover, $F$ is a subspace of $\widetilde{F}$ and $D_{f_0}$ extends to an isometric isomorphism $\widetilde{F} \to L^p\left(X,\mathfrak{B},|f_0|^p\mu\right)$, which we also denote by $D_{f_0}$.

Let $u\in \Isom(F)$. We claim that $u$ can be canonically extended to a surjective, linear isometry $\widetilde{u}$ on $\widetilde{F}$.

Set $g_0=u(f_0)$, which belongs to $\Full(F)$ by Lemma~3.4 in \cite{Har81IsometriesSubLp}.
Then $u$ induces a linear isometry $v$ from the subspace $Q(f_0)$ of $L^p\left(X,\mathfrak{B},|f_0|^p\mu\right)$ to the subspace $Q(g_0)$ of $L^p\left(X,\mathfrak{B},|g_0|^p \mu\right)$.
The maps to be constructed are shown in the following commutative diagram:

\[
\xymatrix@!C{
F \ar@{^{(}->}[dr]  \ar[rr] \ar[dd]_(0.25){u} & & { Q(f_0) } \ar@{^{(}->}[dr] \ar@{->}|(0.5){\hole}[dd]_(0.25){v} \\
 & { \widetilde{F} } \ar[rr]^(0.25){D_{f_0}}_(0.25){\cong} \ar@{-->}[dd]_(0.25){\widetilde{u}}
& & { L^p(\mathfrak{B},|f_0|^p \mu) } \ar@{-->}[dd]^(0.25){\widetilde{v}}
\\
F \ar@{^{(}->}[dr] \ar@{->}|(0.5){\hole}[rr]
&& { Q(g_0) } \ar@{^{(}->}[dr]
\\
& { \widetilde{F} } \ar[rr]^(0.25){D_{g_0}}_(0.25){\cong}
& & { L^p(\mathfrak{B},|g_0|^p \mu) }
\\
}
\]

The constant function $1=\tfrac{f_0}{f_0}$ belongs to $Q(f_0)$, and also to $Q(g_0)$.
It is easy to see that $v$ satisfies $v(1)=1$.
Apply Theorem~2.2 in \cite{Har81IsometriesSubLp} to extend $v$
uniquely to a surjective, linear isometry $\widetilde{v}$ on $\widetilde{F}$.
The desired isometry $\widetilde{u}$ is then given by
\[
\widetilde{u} = D_{g_0}^{-1}\circ\widetilde{v}\circ D_{f_0}.
\]
The claim is proved.

The assignment $u\mapsto \widetilde{u}$ defines an inverse $\varphi$ to the restriction map,
so $\varphi$ is surjective.
It is clear that $\varphi$ is a group homomorphism, and it is easy to see that $\varphi$ is continuous.
It remains to show that $\varphi^{-1}$ is continuous.

We first describe the elements in $\widetilde{F}$.
Let $n\in\N$ and let $\xi\colon\C^n\to\C$ be a measurable map.
If $f_1,\ldots,f_n$ are elements of $F$, then the function
\[
h = \xi\left( \tfrac{f_1}{f_0},\ldots,\tfrac{f_n}{f_0} \right)
\]
is $\mathfrak{B}$-measurable.
Therefore, $f_0 h$ is an element of $\widetilde{F}$ whenever
\[
\int \left| f_0(x)h(x)   \right|^p d\mu(x) <\infty.
\]
It is not hard to check that elements of this form are dense in $\widetilde{F}$.
Finally, if $u\in\Isom(F)$, then its unique extension $\widetilde{u}$ to $\widetilde{F}$ satisfies
\[
\widetilde{u}(f_0h)
= u(f_0) \xi\left( \tfrac{u(f_1)}{u(f_0)},\ldots,\tfrac{u(f_n)}{u(f_0)} \right).
\]

It follows that the assignment $u\mapsto \widetilde{u}(f_0h)$, as a map
$\Isom(F) \to \widetilde{F}$, is measurable, and hence $\varphi^{-1}$ is measurable.

Note that $\Isom(F)$ and $\Isom(\widetilde{F})$ are Polish groups, since $F$ and $\widetilde{F}$ are
separable Banach spaces.
It follows that $\Isom(\widetilde{F},F)$ is a Polish group, being a closed subgroup of $\Isom(\widetilde{F})$.
Since $\varphi^{-1}$ is a measurable group homomorphism between Polish groups, it follows that
$\varphi^{-1}$ is continuous (see, for example, Theorem~9.10 in \cite{Kec95DescriptiveSetThy}).
The proof is finished.
\end{proof}

\begin{cor}
\label{cor:IsomLpSubgp}
Let $p\in[1,\infty)\setminus \{4,6,\ldots\}$.
Then, the isometry group of every separable $\SL^p$-space is topologically isomorphic to a closed subgroup of the isometry group of a separable $L^p$-space.
\end{cor}

\begin{qst}
Does \autoref{cor:IsomLpSubgp} also hold in the nonseparable case?
That is, is the the isometry group of every $\SL^p$-space topologically isomorphic to a closed subgroup of the isometry group of an $L^p$-space?
\end{qst}

We now turn to the construction of the maps $F^p(G)\to F^q(G)$.
When $G$ is second countable, their existence follows easily from the next corollary.
For general $G$, however, said corollary will be used as the base case of an induction argument, the inductive step being \autoref{prp:NormsSigmaCpctLcg}.

\begin{cor}
\label{prp:Norms2CtblLcg}
Let $p\in[1,\infty)\setminus\{4,6,\ldots\}$, let $G$ be a second-countable, locally compact group, and let $f\in L^1(G)$.
Then $\|f\|_{\classSL^p}=\|f\|_{\classL^p}$.
\end{cor}
\begin{proof}
Since $\classL^p\subseteq\classSL^p$, we have $\|f\|_{\classL^p}\leq\|f\|_{\classSL^p}$; see also \autoref{rmk:ExistenceCanonicalMaps}.
By (2) in \autoref{rmk:FpG}, there exist a separable $\SL^p$-space $E$ and an isometric group representation $\rho\colon G\to\Isom(E)$ such that, with $\pi\colon L^1(G)\to\B(E)$ denoting the integrated form of $\rho$, we have
\[
\|f\|_{\classSL^p} = \|\pi(f)\|.
\]
By \autoref{thm:Extension}, there exist an $L^p$-space $\widetilde{E}$ containing $E$ as a closed subspace, and a topological group isomorphism
\[
\varphi\colon\Isom(E)\to\Isom(\widetilde{E},E).
\]
Consider the isometric group representation of $G$ on $\widetilde{E}$ given by
\[
\widetilde{\rho}=\varphi\circ\rho \colon G \to \Isom(\widetilde{E},E)\subseteq\Isom(\widetilde{E}).
\]
Let $\widetilde{\pi}\colon L^1(G)\to \B(\widetilde{E})$ be the integrated form of $\widetilde{\rho}$.

It is easy to see, using that $\widetilde{\rho}(s)$ leaves $E$ invariant for every $s\in G$, that $\widetilde{\pi}(f)$ leaves $E$ invariant as well.
Using this at the second step, we get
\begin{align*}
\| \pi(f) \|
& = \sup \left\{ \|\pi(f)\xi\|_p \colon \xi\in E, \|\xi\|_p\leq 1 \right\} \\
&\leq  \sup \left\{ \|\widetilde{\pi}(f)\xi\|_p \colon \xi\in\widetilde{E}, \|\xi\|_p\leq 1 \right\} \\
&=\| \widetilde{\pi}(f) \|.
\end{align*}
Therefore,
\[
\|f\|_{\classSL^p}
= \|\pi(f)\|
\leq \| \widetilde{\pi}(f) \|
\leq \sup \left\{ \| \tau(f) \| \colon \tau\in\Rep_G(\classL^p) \right\}
= \|f\|_{\classL^p},
\]
as desired.
\end{proof}

\begin{lma}
\label{prp:NormsSigmaCpctLcg}
Let $p\in[1,\infty)\setminus\{4,6,\ldots\}$, let $G$ be a $\sigma$-compact, locally compact group, and let $f\in L^1(G)$.
Then $\|f\|_{\classSL^p}=\|f\|_{\classL^p}$.
\end{lma}
\begin{proof}
We may assume that $\|f\|_1\leq 1$.
As in the proof of \autoref{prp:Norms2CtblLcg}, we have $\|f\|_{\classL^p}\leq\|f\|_{\classSL^p}$.
Let us show the reverse inequality.
By (2) in \autoref{rmk:FpG}, there exists a $\SL^p$-space $E$ and an isometric group representation $\rho\colon G\to\Isom(E)$ whose integrated form $\pi\colon L^1(G)\to\B(E)$ satisfies
\[
\|f\|_{\classSL^p} = \|\pi(f)\|.
\]
Let $\varepsilon>0$.
Choose $\xi_0$ in $E$ with $\|\xi_0\|=1$ such that
\begin{align}
\label{prp:NormsSigmaCpctLcg:eq1}
\|\pi(f)\xi_0\| \geq \|\pi(f)\| - \tfrac{\varepsilon}{2}.
\end{align}
Since $\rho$ is continuous, there exists an open neighborhood $V$ of the identity element $e$ in $G$ such that $\|\rho(s)-\id_E\|<\tfrac{\varepsilon}{2}$ for every $s\in V$.
By the Kakutani-Kodaira Theorem (see, for example, \cite[Theorem~8.7, p.71]{HewRos79AbstrHarmAna1}), there exists a compact, normal subgroup $N$ of $G$ such that $N\subseteq V$ and such that $G/N$ is second-countable.

We consider the fixed point subspace of $E$ for the (restricted) action of $N$ on $E$:
\[
E^N = \{\xi\in E\colon \rho(s)\xi=\xi \text{ for all } s\in N\}.
\]
With $\mu$ denoting the normalized Haar measure on $N$, define an averaging map $P_N\colon E\to E^N$ by
\[
P_N(\xi)= \int_N \rho(s)\xi\ d\mu(s)
\]
for all $\xi\in E$.
Then $P_N$ is contractive and linear.

For every $\eta$ in $E$ with $\|\eta\|\leq 1$, we have
\[
\|P_N(\eta)-\eta\|
=\left\| \int_N \left( \rho(s)\eta -\eta \right)\ d\mu(s)\right\| \leq \int_N \|\rho(s)-\mbox{id}_E\|\cdot\|\eta\| \ d\mu(s)
 <\tfrac{\varepsilon}{2}.
\]
Since $\|\pi(f)\xi_0\|\leq 1$, we deduce that
\begin{align}
\label{prp:NormsSigmaCpctLcg:eq2}
\|P_N(\pi(f)\xi_0)\|\geq \|\pi(f)\xi_0\|-\frac{\varepsilon}{2}.
\end{align}

The isometric representation $\rho\colon G\to\Isom(E)$ induces an isometric representation
$\rho_N\colon G/N \to \Isom(E^N)$ given by
\[
\rho(sN)\xi = \rho(s)\xi
\]
for all $s\in G$ and $\xi\in E$. Let $\pi_N\colon L^1(G/N)\to\B(E^N)$ denote the integrated form of $\rho_N$.
Since $E^N$ is a closed subspace of $E$, it follows that $E^N$ is a $\SL^p$-space.

The map $P_N$ induces a linear map $Q_N\colon\B(E)\to\B(E^N)$, which sends an operator $a\in\B(E)$ to the operator
$Q_N(a)\colon E^N\to E^N$ given by
\[
Q_N(a)\xi=P_N(a\xi)\]
for $\xi\in E^N\subseteq E$.

Consider the map $T_N\colon L^1(G)\to L^1(G/N)$ given by
\[
T_N(g)(sN) = \int_N g(sn)\ d\mu(n).
\]
for $s\in G$.
It is well-known that $T_N$ is a contractive homomorphism; see \cite[Theorem~3.5.4, p.106]{ReiSte00HarmAna}.
It is straightforward to check that the following diagram is commutative:
\[
\xymatrix{
L^1(G) \ar[r]^{\pi} \ar[d]_{T_N}
& \B(E) \ar[d]^{Q_N} \\
L^1(G/N) \ar[r]_{\pi_N}
& \B(E^N).
}
\]
In particular,
\[
\pi_N(T_N(f))P_N(\xi) = P_N(\pi(f)\xi)
\]
for all $\xi\in E$.

Using that $\|P_N(\xi_0)\|\leq 1$ at the second step, using \eqref{prp:NormsSigmaCpctLcg:eq2} at the fourth step, and using \eqref{prp:NormsSigmaCpctLcg:eq1} at the last step, we conclude that
\begin{align*}
\|T_N(f)\|_{\classSL^p}
&\geq \|\pi_N(T_N(f))\|\\
&\geq \|\pi_N(T_N(f))P_N(\xi_0)\| \\
&= \|P_N(\pi(f)\xi_0)\|\\
&\geq \|\pi(f)\xi_0\|-\tfrac{\varepsilon}{2}\\
&\geq \|\pi(f)\|-\varepsilon.
\end{align*}

Since $G/N$ is second-countable, we may apply \autoref{prp:Norms2CtblLcg} at the second step to obtain
\begin{align*}
\|f\|_{\classL^p}
\geq \|T_N(f)\|_{\classL^p}
= \|T_N(f)\|_{\classSL^p}
\geq \|\pi(f)\|-\varepsilon
=\|f\|_{\classSL^p}-\varepsilon.
\end{align*}
Since this holds for all $\varepsilon>0$, we have shown that $\|f\|_{\classSL^p}\leq\|f\|_{\classL^p}$, as desired.
\end{proof}

\begin{thm}
\label{thm:NormsLcg}
Let $p\in[1,\infty)\setminus\{4,6,\ldots\}$, let $G$ be a locally compact group, and let $f\in L^1(G)$.
Then $\|f\|_{\classSL^p}=\|f\|_{\classL^p}$.
\end{thm}
\begin{proof}
It is enough to prove the statement under the additional assumption that $f$ belongs to $C_c(G)$, the algebra of continuous functions $G\to\C$ that have compact support.
In this case, there exists an open (and hence closed) subgroup $H$ of $G$ such that $H$ is $\sigma$-compact and such that the support of $f$ is contained in $H$.

Given a function $g$ in $L^1(H)$, we extend $g$ to a function in $L^1(G)$, also denoted by $g$, by setting $g(s)=0$ for all $s\in G\setminus H$.
This defines an isometric homomorphism $\iota\colon L^1(H)\to L^1(G)$.
Considering the universal completions for representations on $\SL^p$- and $L^p$-spaces, we obtain two contractive homomorphisms
\[
\alpha_\subS\colon F^p_\subS(H) \to F^p_\subS(G),\quad
\alpha\colon F^p(H) \to F^p(G).
\]
Together with the maps $\kappa_\subS^H$ and $\kappa_\subS^G$ from \autoref{nota:kappaSameP}, we obtain the following commutative diagram:
\[
\xymatrix{
F^p_\subS(H) \ar[d]_{\alpha_\subS} \ar[r]^{\kappa_\subS^H}
& F^p(H) \ar[d]^{\alpha} \\
F^p_\subS(G) \ar[r]_{\kappa_\subS^G}
& F^p(G). \\
}
\]
We can consider $f$ as a function in $L^1(H)$ or in $L^1(G)$, and we will denote the norms in the corresponding universal completions by $\|f\|_{\classL^p}^H$, $\|f\|_{\classSL^p}^H$, $\|f\|_{\classL^p}^G$ and $\|f\|_{\classSL^p}^G$.

As in the proof of \autoref{prp:Norms2CtblLcg}, we have $\|f\|_{\classL^p}^G\leq\|f\|_{\classSL^p}^G$.
Using that $\alpha_S$ is contractive at the first step, and applying \autoref{prp:NormsSigmaCpctLcg} for $H$ at the second step, we deduce that
\[
\|f\|_{\classSL^p}^G
\leq \|f\|_{\classSL^p}^H
= \|f\|_{\classL^p}^H.
\]
Thus, in order to obtain the desired inequality $\|f\|_{\classSL^p}^G\leq\|f\|_{\classL^p}^G$, it is enough to show that
\[
\|f\|_{\classL^p}^H \leq \|f\|_{\classL^p}^G.
\]

By (2) in \autoref{rmk:FpG}, there exist a $L^p$-space $E$ and an isometric group representation $\rho\colon H\to\Isom(E)$ whose integrated form $\pi\colon L^1(H)\to\B(E)$ satisfies
\[
\|f\|_{\classL^p}^H = \|\pi(f)\|.
\]
We can induce $\rho$ to an isometric representation of $G$ as follows:
Consider the space
\[
\Ind_H^G(E)
=\left\{ \omega\in\ell^\infty(G,E)\ :\
\begin{matrix}
{\mbox{\; $\rho(h)(\omega(sh))=\omega(s)$ for all $s\in G$ and $h\in H$}}
 \\
{\mbox{and ($sH \mapsto \|\omega(s)\|$) is in $\ell^p(G/H)$}}
\end{matrix}
\right\},
\]
with the norm of an element $\omega\in \Ind_H^G(E)$ given by
\[
\|\omega\| = \left( \sum_{sH\in G/H} \|\omega(s)\|^p \right)^{1/p}.
\]
(The covariance condition $\rho(h)(\omega(sh))=\omega(s)$ ensures that for each $s$ in $S$, the norm $\|\omega(s)\|$ depends only on the class of $s$ in $G/H$.)
Since $G/H$ is discrete, we can choose a section $\sigma\colon G/H \to G$.
By assigning to an element $\omega\in \Ind_H^G(E)\subseteq \ell^\infty(G,E)$ the function $\omega\circ\sigma\colon G/H\to E$, we obtain an isometric isomorphism
\[
\Ind_H^G(E) \xrightarrow{\cong} \ell^P(G/H,E) \cong \bigoplus^p_{G/H} E,
\]
which shows that $\Ind_H^G(E)$ is an $L^p$-space.

The induced representation $\widetilde{\rho}=\Ind_H^G(\rho)\colon G\to\Isom(\Ind_H^g(E))$ is given by
\[
(\widetilde{\rho}(s)\omega)(t) = \omega(s^{-1}t),
\]
for $\omega\in\widetilde{E}$, and $s,t\in G$.
We let $\widetilde{\pi}\colon L^1(G)\to \B(\Ind_H^G(E))$ denote the integrated form of $\widetilde{\rho}$.

Consider the map $\varepsilon\colon E\to\Ind_H^G(E)$ given by
\[
\varepsilon(\xi)(s) =
\begin{cases}
\rho(s^{-1})\xi, & s\in H \\
0, & s\notin H
\end{cases}
\]
for $\xi\in E$. Let $e$ denote the unit element in $G$, and consider the evaluation map $\ev_e\colon \Ind_H^G(E)\to E$.
We have that $\varepsilon$ and $\ev_e$ are linear and contractive, and that $\ev_e\circ\varepsilon = \id_E$.
It follows in particular that $\varepsilon$ defines an isometric embedding of $E$ into $\Ind_H^G(E)$.
We can use $\varepsilon$ and $\ev_e$ to define an isometric homomorphism $Q\colon \B(E)\to\B(\Ind_H^G(E))$ by
\[
Q(a)\omega = \varepsilon(a\ev_e(\omega)) = \varepsilon(a\omega(e))
\]
for $a\in \B(E)$ and $\omega\in \Ind_H^G(E)$. It is then straightforward to check that the following diagram is commutative:
\[
\xymatrix{
L^1(H) \ar[r]^{\pi} \ar[d]^{\iota}
& \B(E) \ar[d]^{Q} \\
L^1(G) \ar[r]^{\widetilde{\pi}}
& \B(\widetilde{E}). \\
}
\]
We a slight abuse of notation, we write $f=\iota(f)$. We conclude that
\begin{align*}
\|f\|_{\classL^p}^G
\leq \|\widetilde{\pi}(f)\|
= \|Q(\pi(f))\|
= \|\pi(f)\|
= \|f\|_{\classL^p}^H,
\end{align*}
which is the desired inequality.
\end{proof}

The above theorem can be used to show that, for some values of $p\in [1,\infty)$,
the canonical maps between certain universal completions are always isometric.

\begin{cor}
\label{cor:QandSisom}
Let $p\in[1,\infty)$ and let $G$ be a \lcg.
\begin{enumerate}
\item
If $p\notin\{4,6,8,\ldots\}$, then the canonical map
\[
\kappa_{\subS}\colon F^p_{\subS}(G) \to F^p(G),
\]
is an isometric isomorphism.
\item
If $p\notin\{\tfrac{4}{3},\tfrac{6}{5},\tfrac{8}{7},\ldots\}$, then the canonical map
\[
\kappa_{\subQ}\colon F^p_{\subQ}(G) \to F^p(G),
\]
is an isometric isomorphism.
\end{enumerate}
\end{cor}
\begin{proof}
The first statement is an immediate consequence of \autoref{thm:NormsLcg}.
The second claim follows from part (1), using duality.
We omit the details.
\end{proof}

\begin{rmk}
It is known that the conclusion of \autoref{thm:Extension} fails for $p\in \{4,6,8,\ldots\}$.
Indeed, in \cite{Lus78ConsequencesRudin}, Lusky shows that the Hardin's extension theorem (Theorem~4.2 in \cite{Har81IsometriesSubLp}) fails for all even integers greater than 2, providing concrete counterexamples that are based on computations of Rudin in Example~3.6 of \cite{Rud76LpIsom}.
However, we do not know whether the restrictions on the H\"older exponent are necessary in \autoref{thm:NormsLcg} and \autoref{cor:QandSisom}.
\end{rmk}

\autoref{cor:QandSisom} suggests the following:

\begin{qst}\label{qst:QSLp}
Let $G$ be a locally compact group and let
\[p\in [1,\infty)\setminus\left\{2n,\frac{2n}{2n-1}\colon n\geq 2\right\}.\]
Then there are canonical isometric isomorphisms $F^p(G)\cong F^p_{\subS}(G)\cong F^p_{\subQ}(G)$.
Is the canonical map $F^p_{\subQS}(G)\to F^p(G)$ an isometric isomorphism?
\end{qst}

Again, the answer to the above question is yes if $p=1,2$, and also if $G$ is amenable (for arbitrary $p$).
One would hope to be able to combine the facts that $F^p(G)\cong F^p_{\subS}(G)$ and $F^p(G)\cong F^p_{\subQ}(G)$
to say something about $F^p_{\subQS}(G)$ in relation to $F^p(G)$, but this is not clear. \autoref{qst:QSLp} may
well have a negative answer.

\vspace{0.3cm}

The following is the main result of this subsection.

\begin{thm}
\label{thm:CanMaps}
Let $G$ be a \lcg.
\begin{enumerate}
\item
If $1\leq p \leq q\leq 2$, then the identity map on $L^1(G)$ extends to a contractive homomorphism
\[
\gamma_{p,q}\colon F^p(G)\to F^q(G)
\]
with dense range.
\item
If $2\leq r \leq s$, then the identity map on $L^1(G)$ extends to a contractive homomorphism
\[
\gamma_{s,r}\colon F^s(G)\to F^r(G)
\]
with dense range.
\end{enumerate}
Moreover, the following diagram is commutative:
\[
\xymatrix{
 { F^p_\subS(G) } \ar[dr] \ar[d]^{\kappa_\subS^p}_{\cong}
& { F^q_\subS(G) } \ar[dr] \ar[d]^{\kappa_\subS^q}_{\cong}
& { F^2_\subS(G) } \ar@{=}[d] \\
{ F^p(G) } \ar@{-->}[r]_{\gamma_{p,q}}
& { F^q(G) } \ar@{-->}[r]_{\gamma_{q,2}}
& { C^*(G) } \ar@{=}[d]
& { F^r(G) } \ar@{-->}[l]_{\gamma_{r,2}}
& { F^s(G) } \ar@{-->}[l]_{\gamma_{s,r}} \\
& & { F^2_\subQ(G) }
& { F^r_\subQ(G) } \ar[ul] \ar[u]^{\kappa_\subQ^r}_{\cong}
& { F^s_\subQ(G) } \ar[ul] \ar[u]^{\kappa_\subQ^s}_{\cong} \\
}
\]
\end{thm}
\begin{proof}
We only prove the first part, since the second one follows from the first using duality. Let $f\in L^1(G)$.
We use \autoref{prop:LqSLp} at the first step and \autoref{cor:QandSisom} at the second to get
\[
\|f\|_{\classL^q}\leq \|f\|_{\classSL^p}=\|f\|_{\classL^p}.
\]
We conclude that the identity map on $L^1(G)$ extends to a contractive homomorphism $\gamma_{p,q}\colon F^p(G)\to F^q(G)$ with dense range.
Commutativity of the diagram depicted in the statement follows from the fact that all the maps involved are the identity on the respective copies of $L^1(G)$.
\end{proof}

We point out that the statement analogous to \autoref{thm:CanMaps} is in general false for \'etale groupoids.
(\'Etale groupoid $L^p$-operator algebras have been introduced and studied in \cite{GarLup14arX:Groupoids}.)
Indeed, the analogs of
Cuntz algebras $\mathcal{O}_n^p$ on $L^p$-spaces (introduced by Phillips in \cite{Phi12arX:LpCuntz}), are
groupoid $L^p$-operator algebras by Theorem~7.7 in \cite{GarLup14arX:Groupoids}. On the other hand,
if $p$ and $q$ are different H\"older
exponents and $n\geq 2$ is an integer, then there is no non-zero continuous homomorphism
$\mathcal{O}_n^p\to \mathcal{O}_n^q$ by Corollary~9.3 in~\cite{Phi12arX:LpCuntz} (see also the comments after it),
which rules out any reasonable generalization of \autoref{thm:CanMaps}
to groupoids. In particular, there seems to be no analog of Hardin's results (or our \autoref{thm:Extension})
for groupoids.

\section{Algebras of \texorpdfstring{$p$}{p}-pseudofunctions and amenability}

In this section, we recall the construction of the algebra $F_\lambda^p(G)$ of $p$-pseudo\-functions on a \lcg{} $G$, for $p\in [1,\infty)$, as introduced by Herz in \cite{Her73SynthSubgps}.
(Our notation differs from the one used by Herz.)
There are natural contractive homomorphisms with dense range from any of the universal completions studied in the previous section, to the algebra of $p$-pseudofunctions.
In \autoref{thm:AmenTFAE}, we characterize amenability of a locally compact group $G$ in terms of these maps. As an application, we show in \autoref{cor:FpG} that for an amenable group $G$, and for $1\leq p\leq q\leq 2$ or for $2\leq q\leq p<\infty$, the canonical map $\gamma_{p,q}\colon F^p(G)\to F^q(G)$ constructed in \autoref{thm:CanMaps} is always
injective, and that it is surjective if and only if $G$ is finite.

\subsection{Algebras of \texorpdfstring{$p$}{p}-pseudofunctions}

We denote by $p$ a fixed H\"older exponent in $[1,\infty)$.
For a \lcg{} $G$, let $\lambda_p\colon L^1(G)\to \B(L^p(G))$ denote the (integrated form of the) left regular representation of $G$ on $L^p(G)$.
For $f$ in $L^1(G)$, we have $\lambda_p(f)\xi=f\ast \xi$ for all $\xi\in L^p(G)$.

\begin{df}
\label{df:reducedGroupAlgebra}
Let $G$ be a \lcg.
The \emph{algebra of $p$-pseudo\-functions} on $G$, here denoted by $F^p_\lambda(G)$, is the completion of $L^1(G)$ with respect to the norm induced by the left regular representation of $L^1(G)$ on $L^p(G)$.
\end{df}

\begin{rmk}
Let $G$ be a \lcg.
The left regular representation $\lambda_p\colon L^1(G)\to \B(L^p(G))$ induces an isometric embedding
$F^p_\lambda(G) \to \B(L^p(G))$ under which we regard $F^p_\lambda(G)$ as represented on $L^p(G)$.
In particular, $F^p_\lambda(G)$ is an $\classL^p$-operator algebra.

In the literature, the elements of $F^p_\lambda(G)$ have been called \emph{$p$-pseudofunctions}, and the Banach algebra $F^p_\lambda(G)$ is usually denoted by $\PF_p(G)$.
This terminology goes back to Herz; see Section 8 in \cite{Her73SynthSubgps}.
(We are thankful to Y.~Choi and M.~Daws for providing this reference.)
Our notation follows one of the standard conventions in $C^*$-algebra theory (\cite{BroOza08Book}).
We warn the reader that $F^p_\lambda(G)$ has also been called the \emph{reduced group $\classL^p$-operator algebra} of $G$, and is sometimes denoted $F^p_{\mathrm{r}}(G)$, for example in \cite{Phi13arX:LpCrProd}.\end{rmk}

It is immediate to check that when $p=2$, the algebra $F^2_\lambda(G)$ agrees with the \emph{reduced group \ca{}} of $G$, which is usually denoted $C^*_\lambda(G)$.

The proof of the following proposition is straightforward and is left to the reader.

\begin{prop}
Let $G$ be a \lcg.
The left regular representation of $L^1(G)$ on $L^p(G)$ is a representation in $\Rep_G(\classL^p)$.
Therefore, the identity map on $L^1(G)$ induces a canonical contractive homomorphism
\[
\kappa\colon F^p(G)\to F^p_\lambda(G)
\]
with dense range.
\end{prop}

We now turn to duality.
Let $G$ be a \lcg, and denote by $\Delta\colon G\to \R$ its modular function.
For $f\in L^1(G)$, let $f^\sharp\colon G\to\C$ be given by $f^\sharp(s) = \Delta(s^{-1}) f(s^{-1})$ for all
$s$ in $G$.

\begin{rmk}
\label{rem:sharp}
For $f$ and $g$ in $L^1(G)$, it is straightforward to check that
\begin{enumerate}
\item
$(f^\sharp)^\sharp=f$;
\item
$(f\ast g)^\sharp= g^\sharp\ast f^\sharp$; and
\item
$\|f^\sharp\|_1=\|f\|_1$ for all $f$ in $L^1(G)$.
\end{enumerate}
In other words, the map $f\mapsto f^\sharp$ defines an isometric anti-automorphism of $L^1(G)$.
It is also immediate to check that if $G$ is unimodular, then $f\mapsto f^\sharp$ also defines an isometric
anti-automorphism of $L^p(G)$ for every $p\in [1,\infty]$.
\end{rmk}

Let $p\in (1,\infty)$, and denote by $p'$ its conjugate exponent. Consider the bilinear paring $L^p(G)\times L^{p'}(G)\to \C$ given by
\[
\langle \xi, \eta \rangle = \int_G \xi(s)\eta(s) ds
\]
for all $\xi\in L^p(G)$ and all $\eta\in L^{p'}(G)$.
It is a standard fact that
\[
\|\xi\|_p = \sup\left\{\left| \langle \xi, \eta \rangle \right|\colon \eta \in L^{p'}(G), \|\eta\|_{p'}=1\right\}
\]
for every $\xi\in L^p(G)$.

\begin{prop}\label{prop:dualityReduced}
Let $G$ be a locally compact group, let $p\in (1,\infty)$, and let $p'$ be its conjugate exponent.
Then
there is a canonical isometric anti-isomorphism $F^p_\lambda(G)\cong F^{p'}_\lambda(G)$,
which is induced by the map $\sharp\colon L^1(G)\to L^1(G)$.
\end{prop}
\begin{proof}
Given $f$ in $L^1(G)$, we claim that $\lambda_{p}(f^\sharp)'=\lambda_{p'}(f)$.
Fix $\xi\in L^p(G)$ and $\eta\in L^{p'}(G)$. Then
\begin{align*}
\langle f^\sharp\ast \xi, \eta \rangle
&= \int_G (f^\sharp\ast \xi)(s) \eta(s) ds \\
&= \int_G \left( \int_G \Delta(t^{-1}) f^\sharp(st^{-1})\xi(t) dt \right) \eta(s) ds \\
&= \int_G \int_G \Delta(t^{-1}) \Delta(ts^{-1}) f(ts^{-1}) \xi(t) \eta(s) dt ds \\
&= \int_G \left( \int_G \Delta(s^{-1}) f(ts^{-1}) \eta(s) ds \right) \xi(t) dt \\
&= \int_G (f\ast \eta)(t) \xi(t) dt \\
&= \langle \xi, f\ast \eta \rangle,\end{align*}
so the claim follows.

It follows that $\|\lambda_{p}(f^\sharp)\|=\|\lambda_{p'}(f)\|$, and hence
the map $\sharp\colon L^1(G)\to L^1(G)$ induces a canonical isometric anti-isomorphism
$F^p_\lambda(G)\to F^{p'}_\lambda(G)$, as desired.\end{proof}

With the notation of the proposition above, we point out that when $p\neq 2$,
we do not seem to get the existence of a canonical isometric isomorphism $F^p_\lambda(G)\cong F^{p'}_\lambda(G)$,
since $\|\lambda_p(f)\|$ and $\|\lambda_p(f^\sharp)\|$ are not in general equal, unless $G$ is abelian (see
\autoref{prop:Gabelian}).
In fact, Herz proved in Corollary~1 of \cite{Her76AsymNormsConv},
that for \emph{every} finite non-abelian group $G$, and for \emph{every} $p\in
(1,\infty)\setminus\{2\}$, there exists $f\in \ell^1(G)$ such that $\|\lambda_p(f)\|_p\neq \|\lambda_{p'}(f)\|_{p'}$.

\subsection{Group and Banach algebra amenability}

Let us recall some facts from functional analysis.
If $E$ is a Banach space and $\xi \in E$, then
\[
\|\xi\| = \sup \left\{ |f(\xi)|\colon f\in E', \|f\|=1 \right\},
\]
We use this to prove the following standard result.

\begin{lma}
\label{lma:HB}
Let $E$ and $F$ be two Banach spaces, and let $\varphi\colon E\to F$ be a bounded linear map.
\begin{enumerate}
\item
The map $\varphi$ has dense image if and only if $\varphi'$ is injective.
\item
The map $\varphi$ is an isometric isomorphism if and only if $\varphi'$ is.
\end{enumerate}
\end{lma}
\begin{proof}
(1). Suppose that $\varphi$ has dense image, and let $f$ in $F'$ satisfy $\varphi'(f)=0$.
Then $\varphi'(f)(\xi)=f(\varphi(\xi))=0$ for all $\xi\in E$. By continuity, we must have
$f(\eta)=0$ for all $\eta$ in $F$, so $f=0$ and $\varphi'$ is injective. Conversely, assume
that $\varphi'$ is injective. In order to arrive at a contradiction, assume that there exists $\eta$ in $F$
not in the closure of $\varphi(E)$. Define a linear functional $g\colon \overline{\varphi(E)}+\C \eta\to \C$
by $g(\varphi(\xi))=0$ for all $\xi$ in $E$, and $g(\eta)=\|\eta\|$. Then $g$ is continuous and $\|g\|=1$, so by
Hahn-Banach there exists a functional $\widetilde{g}\colon F\to \C$ extending $g$ with 
Finally, it is clear that $\varphi'(\widetilde{g})$ is the zero functional on $E$, contradicting
injectivity of $\varphi'$. This shows that $\varphi(E)$ is dense in $F$.
(2).
It is clear that $\varphi'$ is an isometric isomorphism if $\varphi$ is.
Assume now that $\varphi'$ is an isometric isomorphism.
Then $\|\varphi'(f)\|=\|f\|$ for every $f\in F'$, so $\varphi'$ maps the unit ball of $F'$
surjectively onto the unit ball of $E'$.
Given $\xi\in E$, we have
\begin{align*}
\|\varphi(\xi)\|
&= \sup \left\{ |f(\varphi(\xi))|\colon f\in F', \|f\|=1 \right\} \\
&= \sup \left\{ |\varphi'(f)(\xi)|\colon f\in F', \|f\|=1 \right\} \\
&= \sup \left\{ |g(\xi)|\colon g\in E', \|g\|=1 \right\} \\
&= \|\xi\|,
\end{align*}
which shows that $\varphi$ is an isometric embedding of $E$ into $F$.
In particular, the image of $\varphi$ is closed in $F$.
Since the image of $\varphi$ is dense in $F$ by part (1) of this lemma,
we conclude that $\varphi$ is an isometric isomorphism.\end{proof}

The next theorem characterizes amenability of a locally compact group in terms of the canonical maps between its enveloping operator algebras.
The case $p=2$ of the result below is a standard fact in \ca{} theory;
see Theorem~2.6.8 in \cite{BroOza08Book}.

When $\mathcal{E}=\classL^p$, the equivalence between (1) and (2) in the theorem below was independently
obtained by Phillips, whose methods are inspired in $C^*$-algebraic techniques; see \cite{Phi14pre:LpMultDom}.

We denote by $\kappa_u\colon F^p_{\mathrm{QS}}(G)\to F^p_\lambda(G)$ the composition $\kappa\circ \kappa_{\subQS}$.

\begin{thm}
\label{thm:AmenTFAE}
Let $G$ be a \lcg, and let $p\in(1,\infty)$.
Then the following are equivalent:
\begin{enumerate}
\item
The group $G$ is amenable.
\item
With $\mathcal{E}$ denoting any of the classes $\classQSL^p$, $\classSL^p$, $\classQL^p$ or $\classL^p$, the canonical map $F_\mathcal{E}(G)\to F^p_\lambda(G)$ is an isometric isomorphism.
\item
With $\mathcal{E}$ denoting any of the classes $\classQSL^p$, $\classSL^p$, $\classQL^p$ or $\classL^p$, the canonical map $F_\mathcal{E}(G)\to F^p_\lambda(G)$ is a (not necessarily isometric) isomorphism.
\end{enumerate}
\end{thm}
\begin{proof}
We begin by introducing the notation that will be used to prove the equivalences.

Let $p'$ be the dual exponent of $p$.
Let $B_{p'}(G)$ be the $p'$-analog of the Fourier-Stieltjes algebra, as introduced in \cite{Run05ReprQSL}.
By definition, $B_{p'}(G)$ is the set of coefficient functions of representations of $G$ on $\QSL^p$-spaces.
We may think of $B_{p'}(G)$ as a subalgebra of the algebra $C_b(G)$ of bounded continuous functions on $G$,
except that the norm of $B_{p'}(G)$ is not induced by the norm of $C_b(G)$.

Under the canonical identification of $\mathrm{UPF}_p(G)$ and $F^p_{\mathrm{QS}}(G)$ (see \autoref{rmk:FpG}),
Theorem~6.6 in \cite{Run05ReprQSL} provides a canonical isometric isomorphism
\[
F^p_{\subQS}(G)' \cong B_{p'}(G) \subseteq C_b(G).
\]

We now turn to the equivalence between the statements.

(1) implies (2).
It is enough to show that the map $\kappa_u\colon F^p_{\mathrm{QS}}(G)\to F^p_\lambda(G)$
is isometric.
Under the assumption that $G$ is amenable, it follows from Theorems~6.7 in \cite{Run05ReprQSL} that
the dual map $\kappa'\colon F^p_\lambda(G)' \to F^p_{\subQS}(G)'$ of $\kappa$, is an isometric isomorphism.
Indeed, with the notation used there, and writing $\cong$ for isometric isomorphism, we have
\[
F^p_\lambda(G)' \cong \PF_p(G)' \xrightarrow{\cong} B_{p'}(G) \cong F^p_{\subQS}(G)'.
\]
It thus follows from \autoref{lma:HB} that $\kappa_u$ is an isometric isomorphism, as desired.

(2) implies (3). Clear.

(3) implies (1). It is enough to show the result assuming that $\kappa\colon F^p(G)\to F^p_\lambda(G)$ is
an isomorphism.

We regard the dual of $\kappa_u$ as a map $\kappa_u'\colon F^p_\lambda(G)'\to B_{p'}(G)\subseteq C_b(G)$.
By Theorem~4.1 in \cite{NeuRun09ColumnRowQSL}, a locally compact group $G$ is
amenable if and only if the constant function $1$ on $G$ belongs
to the image of $F^p_\lambda(G)'$ in $B_{p'}(G)\subseteq C_b(G)$.
Note that $1$ always belongs to the image of $F^p(G)'$ in
$B_{p'}(G)$, since it is a coefficient function of the trivial
representation (on an $L^{p'}$-space).
Now, if $\kappa\colon F^p(G)\to F^p_\lambda(G)$ is an isomorphism,
then so is the dual map $\kappa'\colon F^p_\lambda(G)'\to F^p(G)'$.
It follows that $1$ is in the image of $F^p_\lambda(G)'$ in $B_p(G)\subseteq
C_b(G)$, and hence $G$ is amenable.
\end{proof}

One must exclude $p=1$ in \autoref{thm:AmenTFAE}, since the canonical maps are \emph{always}
isometric in this case, as was shown in \autoref{prp:F1} and \autoref{rem:poneReduced}.

\autoref{thm:AmenTFAE} raises the following natural questions:

\begin{qst} \label{qst1}
Let $G$ be a locally compact group and let $p\in (1,\infty)$. Is the canonical map $\kappa\colon F^p(G)\to F^p_\lambda(G)$ always surjective?
\end{qst}

The question above has an affirmative answer if either $p=2$ or $G$ is amenable, and there are no known
counterexamples.

If \autoref{qst1} has a negative answer, it would be interesting to explore the following:

\begin{pbm} \label{pbm1}
Let $p\in (1,\infty)\setminus\{2\}$. Characterize those locally compact groups $G$ for which
the canonical map $\kappa\colon F^p(G)\to F^p_\lambda(G)$ is injective.
\end{pbm}

If the answer to
\autoref{qst1} is `yes', then the problem above would have the expected solution: injectivity of
the canonical map $\kappa\colon F^p(G)\to F^p_\lambda(G)$ would be equivalent to amenability, by
the equivalence between (1) and (3) in \autoref{thm:AmenTFAE}, by the Open Mapping Theorem.

Although unlikely, the answer to \autoref{pbm1} could in principle depend on $p$.

\begin{nota}\label{notation}
If $G$ is a locally compact amenable group and $p\in (1,\infty)$, or if $G$ is any locally compact group and $p=1$,
we will write $F^p(G)$ instead of any of $F^p_\subQS(G), F^p_\subS(G), F^p_\subQ(G)$ or $F^p_\lambda(G)$,
since they are isometrically isomorphic by \autoref{thm:AmenTFAE}, \autoref{prp:F1} and \autoref{rem:poneReduced}.
\end{nota}

In the discrete case, amenability of the group is characterized by amenability of any of its associated universal
enveloping algebras.

\begin{thm}
\label{thm:FpGamenableIFF}
Let $G$ be a locally compact group and let $p\in [1,\infty)$.
Let $\mathcal{E}$ be any of the classes $\classQSL^p$, $\classQL^p$, $\classSL^p$, or $\classL^p$.
If $G$ is amenable, then so is $F_\mathcal{E}(G)$. The converse is true if $G$ is discrete.
\end{thm}
\begin{proof}
It is a well known result due to B. Johnson (see \cite{Joh72CohomologyBAlg}) that $G$ is amenable if and only if the group algebra $L^1(G)$ is amenable (even if $G$ is not discrete).
If $G$ is amenable, then so is $L^1(G)$, and hence also $F_\mathcal{E}(G)$ by Proposition~2.3.1 in \cite{Run02BookAmen}, since the image of $L^1(G)$ in $F_\mathcal{E}(G)$ is dense.

Conversely, suppose that $F_\mathcal{E}(G)$ is amenable and that $G$ is discrete.
Then $F^p(G)$ is amenable again by Proposition~2.3.1 in \cite{Run02BookAmen}, because there is a contractive homomorphism $F_\mathcal{E}(G)\to F^p(G)$ with dense range.
Another use of Proposition~2.3.1 in \cite{Run02BookAmen}, this time with the map $\gamma_{p,2}\colon F^p(G)\to F^2(G)=C^*(G)$ constructed in \autoref{thm:CanMaps}, shows that $C^*(G)$ must be amenable in this case.
Now Theorem~2.6.8 in \cite{BroOza08Book} implies that $G$ is amenable.
\end{proof}

The following question naturally arises:

\begin{qst}
\label{qst:JohnsonGeneralized}
Does amenability of $F^p(G)$, for $p\neq 2$, characterize amenability of $G$ in full generality? \end{qst}

For $p=1$, the answer is yes, by Johnson's Theorem (\cite{Joh72CohomologyBAlg}). The result is known to be false for $p=2$.
Indeed, Connes proved in \cite{Con76ClassifInjFactors} that if $G$ is a connected Lie group, then $C^*(G)$ (and hence $C^*_\lambda(G))$ is amenable.
However, there are non-amenable connected Lie groups, such as $\mbox{SL}_2(\R)$ (whose group $C^*$-algebra is even type I).
We do not know, however, whether $F^p(\mbox{SL}_2(\R))$ is amenable for $p\neq 1,2$.
\vspace{5pt}

We close this subsection with the computation of the maximal ideal space of $F^p(G)$ when $G$ is an
abelian locally compact group. This result, in this form, will be crucial in the proof of Theorem~3.4 in
\cite{GarThi14arX:Functoriality}. The result is almost certainly well-known, but we have not been able to find a
reference in the literature.

\begin{prop}
\label{prop:MaxIdSp}
Let $G$ be an abelian locally compact group, and let $p\in [1,\infty)$. Then there is a canonical homeomorphism $\phi\colon \Max(F^p(G))\to \widehat{G}$.
Moreover, the Gelfand transform
$$\Gamma_p\colon F^p(G)\to C_0(\widehat{G})$$
is injective, contractive and has dense range.
\end{prop}
\begin{proof}
It is well-known that $\Max(L^1(G))$ is canonically homeomorphic to $\widehat{G}$, and that
the following diagram is commutative:
\begin{eqnarray*} \xymatrix{
L^1(G) \ar[rr] \ar[d]_-{\Gamma_1}&& C^*(G)\ar[d]^-{\Gamma_2}\\
C_0(\widehat{G})&& C_0(\widehat{G})\ar[ll]^-{\mathrm{id}_{C_0(\widehat{G})}}.
}\end{eqnarray*}

Set $X=\Max(F^p(G))$.
The canonical map $\gamma_{p,2}\colon F^p(G)\to C^*(G)$ induces a continuous function $\phi\colon \widehat{G}\to X$. We claim that $\phi$ is injective.
Indeed, $\phi$ is the restriction of $\gamma_{p,2}'\colon C^*(G)'\to F^p(G)'$ to the multiplicative linear functionals of norm one.
Since $\gamma_{p,2}'$ is injective by part (1) of \autoref{lma:HB}, the claim follows.

A similar argument shows that the canonical inclusion $L^1(G)\to F^p(G)$ induces an injective continuous function $\psi\colon X\to \widehat{G}$.
By naturality of the Gelfand transform, we must have $\psi\circ\phi=\id_{\widehat{G}}$, showing that $\widehat{G}$
and $X$ are homeomorphic.

For the last claim, it only remains to show that $\Gamma_p$ is injective and has dense range.
Injectivity follows from Theorem~4 in Section~1.5 of \cite{Der11ConvOps}. Density of its range follows from the
facts that $\gamma_{p,2}$ has dense range by \autoref{thm:CanMaps}, that $\Gamma_2$ is an isometric isomorphism, and
that the diagram below commutes:
\begin{eqnarray*} \xymatrix{
L^1(G) \ar[r] \ar[d]_-{\Gamma_1}& F^p(G)\ar[r]^-{\gamma_{p,2}}\ar[d]^-{\Gamma_p}& C^*(G)\ar[d]^-{\Gamma_2}\\
C_0(\widehat{G})& C_0(\widehat{G})\ar[l]^-\cong&C_0(\widehat{G})\ar[l]^-\cong .
}\end{eqnarray*}
\end{proof}

\subsection{Canonical maps $F^p(G)\to F^q(G)$ revisted}

In this subsection, we will use \autoref{thm:AmenTFAE} to obtain further information about the maps
$\gamma_{p,q}$ constructed in \autoref{thm:CanMaps}, by regarding them as maps between the algebras
of $p$- and $q$-pseudofunctions; see \autoref{cor:FpG}.

We begin with a general discussion (see also Section 8 in \cite{Her73SynthSubgps}).

\begin{df}
Let $E$ be a reflexive Banach space.
It is well-known that the Banach algebra $\B(E)$ is the Banach space dual of the projective
tensor product $E\widehat{\otimes}E'$.
The weak*-topology inherited by $\B(E)$ from this identification is usually called the
\emph{ultraweak topology} on $\B(E)$.
\end{df}

Given a Banach space $E$, we write $\langle\cdot,\cdot\rangle_{E,E'}$ for the canonical pairing $E\times E'\to\C$.
Given $\xi\in E$ and given $\eta\in E'$, we write $\xi\otimes\eta$ for the simple tensor product in
$E\widehat{\otimes}E'$.
Regarding an operator $a\in \B(E)$ as a functional on $E\widehat{\otimes}E'$,
the action of $a$ on $\xi\otimes \eta$ is given by
\begin{align}
\label{eq1}
\langle a,\xi\otimes\eta\rangle_{(E\widehat{\otimes}E')',E\widehat{\otimes}E'}
= \langle a\xi,\eta \rangle_{E,E'}.
\end{align}

\begin{df}\label{df:UWKconvergence}
Let $E$ be a Banach space.
Let $(a_j)_{j\in J}$ be a net of operators in $\B(E)$, and let $a\in \B(E)$ be another operator.
We say that $(a_j)_{j\in J}$ \emph{converges ultraweakly} to $a$, if for every $x\in E\widehat{\otimes}E'$
we have
\[\lim_{j\in J} \left|\langle a_j,x\rangle - \langle a,x\rangle\right|=0.\]
\end{df}

The ultraweak topology on $\B(E)$ should not be confused with its weak operator topology.
By definition, a net $(a_j)_{j\in J}$ in $\B(E)$ converges in the weak operator topology to an operator $a\in \B(E)$
if for every $\xi\in E$ and every $\eta\in E'$, we have
\[\lim_{j\in J} \left|\langle a_j\xi ,\eta\rangle - \langle a\xi,\eta\rangle\right|.\]

\begin{rmk}
Since a pair $(\xi,\eta)\in E\times E'$ defines an element $x=\xi\otimes\eta$ in $E\widehat{\otimes}E'$, it follows from \eqref{eq1}, that the ultraweak topology is stronger than the weak operator topology. On the other hand, it is well-known
that the ultraweak topology and the weak operator topology agree on (norm) bounded subsets of $\B(E)$.
\end{rmk}

The following class of Banach algebras will be needed in the proof of \autoref{thm:CanMapsRed}.

\begin{df}\label{df:pseudomeasures}
Let $G$ be a locally compact group and let $p\in [1,\infty)$.
\begin{enumerate}
\item
The algebra of \emph{$p$-pseudomeasures} on $G$, denoted
$PM_p(G)$, is the ultraweak closure of $F^p_\lambda(G)$ in $\B(L^p(G))$.
\item
The algebra of \emph{$p$-convolvers} on $G$, denoted $CV_p(G)$, is the bicommutant of $F^p_\lambda(G)$
in $\B(L^p(G))$.
\end{enumerate}
\end{df}

Algebras of pseudomeasures and convolvers on groups have been thoroughly studied since their
introduction by Herz in Section 8 in \cite{Her73SynthSubgps}.
It is clear that $PM_p(G)\subseteq CV_p(G)$, and it is conjectured that they are equal for every locally
compact group and every H\"older exponent $p\in [1,\infty)$. The reader is referred to \cite{DawSpr14arX:ApproxPropConvPM}
for a more thorough description of the problem, as well as for the available partial results.

\begin{thm}\label{thm:CanMapsRed}
Let $G$ be a locally compact group, and let $p,q\in [1,\infty)$ with
either $p\leq q\leq 2$ or $2\leq q\leq p$.
Assume that
\begin{align}\label{eqn:ineq}\|\lambda_q(f)\|_q \leq \|\lambda_p(f)\|_p\end{align}
for every $f\in L^1(G)$.
Then the identity map on $L^1(G)$ extends to a contractive map
\[
\gamma^\lambda_{p,q}\colon F^p_\lambda(G)\to F^q_\lambda(G),
\]
with dense range. Moreover,

\begin{enumerate}
\item
The map $\gamma_{p,q}^\lambda$ is injective.
\item
Suppose that $CV_p(G)=PM_p(G)$ and $CV_q(G)=PM_q(G)$. If $p\neq q$, then
$\gamma_{p,q}^\lambda$ is not surjective unless $G$ is finite.
\end{enumerate}
\end{thm}

We emphasize that the assumptions $CV_p(G)=PM_p(G)$ and $CV_q(G)=PM_q(G)$ in part (2)
of this theorem, are conjecturally not a restriction; see \cite{DawSpr14arX:ApproxPropConvPM}.

\begin{proof}
We show first the existence of $\gamma^\lambda_{p,q}$.
Let $a$ be an operator in $F^p_\lambda(G)$ and choose
a sequence $(f_n)_{n\in\N}$ in $L^1(G)$ such that $\lim\limits_{n\to\infty}\left\|a-\lambda_p(f_n)\right\|=0$.
Then $(\lambda_p(f_n))_{n\in\N}$ is a Cauchy sequence in $\B(L^p(G))$, and hence
$(\lambda_q(f_n))_{n\in\N}$ is a Cauchy sequence in $\B(L^q(G))$ as well, by the inequality in (\ref{eqn:ineq}).
Set
\[
\gamma^\lambda_{p,q}(a) = \lim_{n\to\infty} \lambda_q(f_n).
\]

It is straightforward to check that this definition is independent of the choice of the sequence
$(f_n)_{n\in\N}$ in $L^1(G)$.
Moreover, it is clear that the resulting homomorphism $\gamma^\lambda_{p,q}\colon F_\lambda^p(G)\to F^q_\lambda(G)$
is contractive and has dense range.

(1). Let us show that $\gamma^\lambda_{p,q}$ is injective.
Fix $a\in F_\lambda^p(G)$ and choose a sequence $(f_n)_{n\in\N}$ in $L^1(G)$ such that
$\lambda_p(f_n)$ converges to $a$ in $\B(L^p(G))$.
Assume that $\gamma^\lambda_{p,q}(a)=0$, so that $(\lambda_q(f_n))_{n\in\N}$ converges to the zero operator on $L^q(G)$.
In order to arrive at a contradiction, suppose that $a\neq 0$. Choose $\xi\in C_c(G)$ such that
$a\xi\neq 0$. Set $\eta=a\xi$. Since
\[\lim_{n\to\infty}\|f_n\ast \xi-\eta\|_p=0,\]
upon passing to a subsequence, we may assume that $(f_n\ast \xi)_{n\in\N}$ converges pointwise
almost everywhere to $\eta$.

Since $(\lambda_q(f_n))_{n\in\N}$ converges to zero in $\B(L^q(G))$, it follows that $(\lambda_q(f_n)\xi)_{n\in\N}$
converges to zero in $L^q(G)$.
Again, upon passing to a subsequence, we may assume that $f_n\ast\xi$ converges pointwise
almost everywhere to zero.
This clearly implies that $\eta=0$ almost everywhere on $G$, which is a contradiction.
We deduce that $a=0$, and hence $\gamma^\lambda_{p,q}$ is injective.

Before proving part (2), let us show that $\gamma^\lambda_{p,q}$ extends to a map
\[\delta_{p,q}\colon PM_p(G)\to PM_q(G)\]
between the ultraweak closures of $F^p_\lambda(G)$ and
$F^q_\lambda(G)$ in $\B(L^p(G))$ and $\B(L^q(G))$, respectively.
The existence of such a map is well-known to the experts, so we only sketch its construction.

Let $a$ be an operator in $PM_p(G)$, and choose a
net $(f_j)_{j\in J}$ in $C_c(G)$ such that $(\lambda_p(f_j))_{j\in J}$ converges to $a$ in the ultraweak topology.
Since the sequence $(\lambda_p(f_j))_{j\in J}$ converges ultraweakly, it follows that it is norm-bounded.
Since $L^p(G)$ is a separable Banach space, the ultraweak topology is metrizable on bounded subsets of $\B(L^p(G))$,
and hence there is a sequence $(f_n)_{n\in\N}$ in $C_c(G)$ such that $\lambda_p(f_n)$ converges ultraweakly to $a$.

By the inquality in (\ref{eqn:ineq}), the sequence $(\lambda_q(f_n))_{n\in\N}$ is norm-bounded in $\B(L^q(G))$.
By the Banach-Alaoglu Theorem, the sequence $(\lambda_q(f_n))_{n\in\N}$ has an ultraweak limit point,
so there is a subsequence $(\lambda_q(f_{n_k}))_{k\in\N}$ that converges ultraweakly to an operator $b\in \B(L^q(G))$.
By construction, $b$ belongs to $PM_q(G)$.
One can show that $b$ does not depend on the choices made, so we set $\delta_{p,q}(a)=b$.

The resulting homomorphism $\delta_{p,q}\colon PM_p(G)\to PM_q(G)$ is easily seen to be contractive and to have
dense range.

(2).
If $G$ is finite, then $\gamma_{p,q}^\lambda$ is clearly surjective,
since it has dense range and $F^q_\lambda(G)$ is
finite dimensional. Conversely, assume that $G$ is infinite. We will show that $\gamma^\lambda_{p,q}$ is not surjective
using results from \cite{CowFou76InclNonincl}.

To reach a contradiction, assume that $\gamma^\lambda_{p,q}$ is surjective.
It follows from the Open Mapping Theorem and part (1) of this theorem, that $\gamma^\lambda_{p,q}$
is an isomorphism (although not necessarily isometric).
This means that there is a constant $K>0$ such that
\[
\| \lambda_p(f) \|_p \leq K \| \lambda_q(f) \|_q,
\]
for every $f\in L^1(G)$.

We claim that $\delta_{p,q}$ is also surjective.
Given $b\in PM_q(G)$,
choose a sequence $(f_n)_{n\in\N}$ in $C_c(G)$ such that the sequence of operators
$(\lambda_q(f_n))_{n\in\N}$ in $\B(L^q(G))$ converges ultraweakly to $b$.
Then the sequence $(\lambda_q(f_n))_{n\in\N}$ is norm-bounded in $\B(L^q(G))$.
It follows that
\[
\sup_{n\in\N} \| \lambda_p(f_n) \| \leq K \sup_{n\in\N} \| \lambda_q(f_n) \|<\infty,
\]
so $(\lambda_p(f_n))_{n\in\N}$ is norm-bounded in $\B(L^p(G))$ as well. By the Banach-Alaoglu Theorem,
there exists a subsequence $(\lambda_p(f_{n_k}))_{k\in\N}$ that converges ultraweakly to an
operator $a\in\B(L^p(G))$. By construction, $a$ belongs to $PM_p(G)$.

We claim that $\delta_{p,q}(a)=b$.
For $\xi,\eta,f \in C_c(G)$, we have
\[
\langle \lambda_p(f)\xi,\eta \rangle_{L^p(G),L^{p'}(G)}
=\int_G\int_G f(st^{-1})\xi(t)\eta(s) \ dsdt
= \langle \lambda_q(f)\xi,\eta \rangle_{L^q(G),L^{q'}(G)}.
\]
We deduce that
\[
\langle a\xi,\eta \rangle
= \lim_{k\to\infty} \left\langle \lambda_p(f_{n_k})\xi,\eta \right\rangle
= \lim_{k\to\infty} \left\langle \lambda_q(f_{n_k})\xi,\eta \right\rangle
= \lim_{n\to\infty} \left\langle \lambda_q(f_n)\xi,\eta \right\rangle
= \langle b\xi,\eta \rangle.
\]
Using that the weak operator topology agrees with the ultraweak topology on bounded sets, 
it follows from the definition that $\delta_{p,q}(a)=b$, as desired.

By our assumption $PM_p(G)=CV_p(G)$ and $PM_q(G)=CV_q(G)$, we can regard $\delta_{p,q}$
as a map between the respective $p$- and $q$-convolvers on $G$.
Now, the fact that $\delta_{p,q}\colon CV_p(G)\to CV_q(G)$ is surjective
contradicts Theorem~2 in \cite{CowFou76InclNonincl}
(where $CV_p(G)$ and $CV_q(G)$ are denoted by $L_p^p(G)$ and $L^q_q(G)$, respectively).
This contradiction implies that $\gamma^\lambda_{p,q}$ is not surjective, as desired.
\end{proof}

\begin{rmk}
\label{rem:CanMapsRedNExist}
Let $G$ be a locally compact group and let $p,q\in [1,\infty)$ satisfy
$\left|\tfrac{1}{q}-\tfrac{1}{2}\right|<\left|\tfrac{1}{p}-\tfrac{1}{2}\right|$.
We point out that even though the map $\gamma_{p,q}\colon F^p(G)\to F^q(G)$ constructed in
\autoref{thm:CanMaps} exists in full generality, the map $\gamma_{p,q}^\lambda\colon F_\lambda^p(G)\to
F_\lambda^q(G)$ may fail to exist for some groups and some exponents $p$ and $q$, since there may be no
constant $M>0$ such that
\[\|\lambda_q(f)\|_q\leq M\|\lambda_p(f)\|_p\]
holds for all $f\in L^1(G)$. Indeed, as it is shown in \cite{Pyt81RadialFctsFreeGps} and \cite{Pyt84ConstrConvFreeGps},
this is the case for any noncommutative free group, and for any exponents $p,q\in (1,\infty)$ with $p\neq q$.
\end{rmk}

In contrast to what was pointed out in \autoref{rem:CanMapsRedNExist},
we have that the assumptions of \autoref{thm:CanMapsRed} are satisfied in a number of situations, particularly
if $G$ is amenable. We state this explicit in the following corollary.

\begin{cor}
\label{cor:FpG}
Let $G$ be an amenable locally compact group, and let $p,q\in [1,\infty)$.
Denote by $\lambda_p\colon L^1(G)\to \B(L^p(G))$
and $\lambda_q\colon L^1(G)\to \B(L^q(G))$ the corresponding left
regular representations of $G$. If either $p\leq q\leq 2$ or $2\leq q\leq p$, then
\[\|\lambda_q(f)\|_q \leq \|\lambda_p(f)\|_p\]
for every $f$ in $L^1(G)$. Moreover,
\begin{enumerate}
\item
The map $\gamma_{p,q}$ constructed in \autoref{thm:CanMaps} is injective.
\item
If $p\neq q$, then $\gamma_{p,q}$ is surjective if and only if $G$ is finite.
\end{enumerate}
\end{cor}
\begin{proof}
It is enough to assume that $1\leq p\leq q\leq 2$, since the other case can then be deduced by using duality.
Given $f\in L^1(G)$, we use \autoref{thm:AmenTFAE} at the first and third step, and \autoref{prop:LqSLp} at
the second step, to get
\[\|\lambda_q(f)\|_q=\|f\|_{\classL^q}\leq \|f\|_{\classL^p}=\|\lambda_p(f)\|_p,\]
as desired.

Part (1) follows immediately from part (1) of \autoref{thm:CanMapsRed}, since
$\gamma_{p,q}=\gamma_{p,q}^\lambda$ by amenability of $G$. Likewise, part (2) follows
from part (2) of \autoref{thm:CanMapsRed}, together with Herz's result (Theorem~5 in \cite{Her73SynthSubgps})
that $PM_r(G)=CV_r(G)$ for all $r\in [1,\infty)$ whenever $G$ is amenable.
\end{proof}

\begin{rmk}
\label{rem:ThmCofHerz}
Adopt the notation from the statement of \autoref{cor:FpG}. The fact that $\|\lambda_q(f)\|_q\leq\|\lambda_p(f)\|_p$ was
announced as Theorem~C in \cite{Her71pSpApplConv} (see the corollary on page~512 of~\cite{HerRiv72EstimatesMixNorm}
for a proof).
\end{rmk}

It is a consequence of \autoref{cor:FpG} that the Banach algebras $F^p(\Z)$ for $p\in [1,2]$,
are pairwise not \emph{canonically} isometrically isomorphic. (Here `canonical' means via a map
which is the identity on $\ell^1(\Z)$.) In a separate paper (\cite{GarThi14arX:Functoriality}), we
show that for $p,q\in [1,2]$, the algebras $F^p(\Z)$ and $F^q(\Z)$ are not
even \emph{abstractly} isometrically isomorphic; see Theorem~3.4 in \cite{GarThi14arX:Functoriality}.
\vspace{5pt}

We point out that when $G$ is abelian, the unnumbered claim in \autoref{cor:FpG} (see also
\autoref{rem:ThmCofHerz}) can be proved much more directly.
We do so in the proposition below, which is well-known to the experts and appears implicitly in the literature.

\begin{prop}\label{prop:Gabelian}
Let $G$ be an abelian locally compact group, and let $f\in L^1(G)$.
\begin{enumerate}
\item
Let $p\in (1,\infty)$. Then $\|\lambda_p(f)\|_{p}=\|\lambda_{p'}(f)\|_{p'}$.
\item
Let $p,q\in [1,\infty)$, and suppose that either $p\leq q\leq 2$ or $2\leq q\leq p$. Then
$\|\lambda_q(f)\|_{q}\leq \|\lambda_{p}(f)\|_{p}$.
\end{enumerate}
\end{prop}
\begin{proof}
(1). Let $f\in L^1(G)$.
It was shown in the proof of \autoref{prop:dualityReduced} that
$\|\lambda_p(f^\sharp)\|_p=\|\lambda_{p'}(f)\|_{p'}$. Since $G$ is unimodular, we have $\|\xi^\sharp\|_p=\|\xi\|_p$
for all $\xi\in L^p(G)$. Using this fact at the third step, we get
\begin{align*} \|\lambda_p(f^\sharp)\|_p&=\sup_{\xi\in L^p(G), \xi\neq 0} \tfrac{\|f^\sharp\ast\xi\|_p}{\|\xi\|_p}\\
&=\sup_{\xi\in L^p(G), \xi\neq 0} \tfrac{\|(\xi^\sharp\ast f)^\sharp\|_p}{\|\xi\|_p}\\
&=\sup_{\xi\in L^p(G), \xi\neq 0} \tfrac{\|\xi\ast f\|_p}{\|\xi\|_p}.\end{align*}

Since $G$ is abelian, we have $\xi\ast f=f\ast\xi$ for all $\xi\in L^p(G)$, and the result follows.

(2).
Denote by $\gamma\colon (1,\infty)\to \R$ the function given by $\gamma(r)=\|\lambda_r(f)\|_r$ for all
$r\in (1,\infty)$. Then $\gamma(r)=\gamma(r')$ for all $r\in (1,\infty)$ by part (1), and $\gamma$ is log-convex by the Riesz-Thorin
Interpolation Theorem. It follows that $\gamma$ has a minimum at $r=2$, and that it is decreasing on $[1,2]$ and increasing
on $[2,\infty)$. This finishes the proof.\end{proof}

\section{\texorpdfstring{$L^p$}{Lp}-crossed products and amenability}

N.~C.~Phillips has announced that for $p\in [1,\infty)$, if $G$ is an amenable discrete group and
$\alpha\colon G\to\Aut(A)$ is an isometric action of $G$ on an $\classL^p$-operator algebra $A$, then the canonical
contractive map $\kappa\colon F^p(G,A,\alpha)\to F^p_\lambda(G,A,\alpha)$ is an isometric isomorphism. (The reader is
referred to \cite{Phi13arX:LpCrProd} for the construction of full and reduced crossed products, as well
as for the definition of the canonical map.) The proof is
analogous to the $C^*$-algebra case, and is likely to appear in a second version of the preprint
\cite{Phi13arX:LpCrProd}.

In analogy with \autoref{thm:AmenTFAE} and \autoref{thm:FpGamenableIFF}, one may be tempted to
conjecture that for $p>1$, amenability of a discrete group $G$ may be equivalent to the canonical map
$\kappa\colon F^p(G,X)\to F^p_\lambda(G,X)$ being an isometric isomorphism for
every locally compact Hausdorff $G$-space $X$, and that this in turn should be equivalent to
amenability of $F^p(G,X)$. (\autoref{thm:AmenTFAE} and \autoref{thm:FpGamenableIFF} are the case
$X=\ast$ of this statement.)
There are many examples that show that this is not true when $p=2$, but one may hope that this holds in all other cases,
because of the extra rigidity of $\classL^p$-operator crossed  (see also \autoref{qst:JohnsonGeneralized}).

However, the statement fails for every
$p\in (1,\infty)\setminus\{2\}$, and we will devote this subsection to the construction of a family of counterexamples.

A crucial notion in our proof of \autoref{thm:ComputationCP} is that of incompressible Banach algebras, which is
due to Phillips (\cite{Phi14pre:LpIncompr}).

\begin{df} Let $p\in [1,\infty)$ and let $A$ be a Banach algebra. We say that $A$ is
\emph{$p$-incompressible} if for every $L^p$-space $E$, every contractive, injective homomorphism
$\rho\colon A\to \B(E)$ is isometric.

If the H\"older exponent $p$ is clear from the context, we will just say that $A$ is incompressible.
\end{df}

Examples of $p$-incompressible Banach algebras include $\B(\ell^p_n)$ for $n\in\N$ (see Theorem 7.2 in
\cite{Phi12arX:LpCuntz}, where $\B(\ell^p_n)$ is denoted $M^p_n$), the analogs $\mathcal{O}_d^p$ of Cuntz algebras on $L^p$-spaces (see Corollary 8.10 in \cite{Phi12arX:LpCuntz}), and \ca s.

The next lemma asserts that a direct limit of $p$-incompressible Banach algebras is again
$p$-incompressible.

\begin{lma}\label{lem: DirLimIncomp}
Let $((A_\mu)_{\mu\in \Lambda},(\varphi_{\mu,\nu})_{\mu,\nu\in \Lambda})$ be a direct limit of Banach algebras with injective, contractive maps $\varphi_{\mu,\nu}\colon A_\mu\to A_\nu$ for all $\mu$ and $\nu$ in $\Lambda$ with $\mu\leq \nu$, and denote by $A$ its direct limit. Let $p\in [1,\infty)$. If $A_\mu$ is
$p$-incompressible for all $\mu$ in $\Lambda$, then so is $A$.
\end{lma}
\begin{proof} For $\mu$ in $\Lambda$, denote by $\varphi_{\mu,\infty}\colon A_\mu \to A$ the canonical contractive homomorphism into the direct limit algebra. Let $E$ be an $L^p$-space, and let
$\rho\colon A\to \B(E)$ be a contractive injective representation. Given $\mu$ in $\Lambda$, the representation
$\rho\circ\varphi_{\lambda,\infty}\colon A_\mu\to \B(E)$
is injective and contractive. Since $A_\mu$ is incompressible, it follows that $\rho\circ\varphi_{\mu,\infty}$ is isometric. Hence, for $a$ in $A_\mu$, we have
$$\|\varphi_{\mu,\infty}(a)\|\leq\|a\|_{A_\mu}=\|(\rho\circ\varphi_{\mu,\infty})(a)\|\leq \|\varphi_{\mu,\infty}(a)\|,$$
ant thus $\|(\rho\circ\varphi_{\mu,\infty})(a)\|= \|\varphi_{\mu,\infty}(a)\|$.
We conclude that $\rho|_{\varphi_{\mu,\infty}(A_\mu)}$ is isometric for all $\mu$ in $\Lambda$. Hence $\rho$ is isometric, and $A$ is incompressible.\end{proof}

We are now ready to show that for any discrete group $G$ (amenable or not) and for any $p\in [1,\infty)$, there exists
a locally compact Hausdorff $G$-space $X$ such that the canonical map
$\kappa\colon F^p(G,X)\to F^p_\lambda(G,X)$ is isometric and such that $F^p(G,X)$ is amenable.
In particular, the analog of \autoref{thm:AmenTFAE},
where one considers actions of $G$ on arbitrary topological spaces other than the one point space,
is false.
\vspace{5pt}

Recall (see, for example, \cite{Ber72HermitianProj}) that if $A$ is a Banach algebra, an element $a\in A$ is said to be
\emph{hermitian} if $\|e^{ita}\|=1$ for all $t\in \R$.
If $X$ is a locally compact Hausdorff space, then every idempotent in $C_0(X)$ is hermitian.

The terminology ``spatial partial isometry" and the notion of spatial system are borrowed from Definition~6.4 in
\cite{Phi12arX:LpCuntz}, and the notation $\overline{M}^p_G$ is taken from Example~1.6 in \cite{Phi12arX:LpCuntz}.

\begin{thm} \label{thm:ComputationCP}
Let $G$ be a discrete group, and let $\alpha\colon G\to \Aut(C_0(G))$ be the isometric
action induced by left translation of $G$ on itself. Let $p\in [1,\infty)$. Then there are natural isometric
isomorphisms
$$\xymatrix{ F^p(G,C_0(G),\alpha)\ar[r]^\kappa & F^p_\lambda(G,C_0(G),\alpha) \ar[r]& \overline{M}^p_G}.$$
Moreover, the right-hand side equals $\K(\ell^p(G))$ when $p>1$, and is strictly smaller than $\K(\ell^1(G))$ when $p=1$
and $G$ is infinite. \end{thm}
Note that $\overline{M}^p_G$, being the direct limit of matrix algebras, is amenable even if $G$ is not.
\begin{proof} The last claim follows from Corollary 1.9 and Example 1.10 in \cite{Phi13arX:LpCrProd}.

We begin by showing that there is a natural isometric isomorphism
$$F^p(G,C_0(G),\alpha) \cong  \overline{M}^p_G.$$
For $s\in G$, let $u_s$ be the standard invertible isometry implementing $\alpha_s$ in the crossed product, and let $\delta_s\in C_0(G)$ be the function $\chi_{\{s\}}$. Then
$\alpha_s(\delta_t)=\delta_{st}$ for all $s,t\in G$, and $\mbox{span}\left(\{\delta_s\colon s\in G\}\right)$ is dense in $C_0(G)$.

For $s,t\in G$, set $a_{s,t}=\delta_s u_{st^{-1}}\in \ell^1(G,C_0(G),\alpha)$.
For $s_1,s_2,t_1,t_2\in G$, we have
\begin{align*} a_{s_1,t_1}a_{s_2,t_2} &= \delta_{s_1}u_{s_1t_1^{-1}}\delta_{s_2}u_{s_2t_2^{-1}}\\
&=\delta_{s_1}\alpha_{s_1t_1^{-1}}(\delta_{s_2})u_{s_1t_1^{-1}}u_{s_2t_2^{-1}} \\
&= \delta_{s_1}\delta_{s_1t_1^{-1}s_2}u_{s_1t_1^{-1}s_2t_2^{-1}}.\end{align*}
Thus, if $s_2\neq t_1$, then $a_{s_1,t_1}a_{s_2,t_2}=0$, because in this case $\delta_{s_1}\delta_{s_1t_1^{-1}s_2}=0$.
Taking $s_2=t_1$, we get $a_{s_1,t_1}a_{t_1,t_2}=a_{s_1,t_2}$.
Hence the elements
$\{a_{s,t}\colon s,t\in G\}$ satisfy the relations for a system of matrix units indexed by $G$. Also, $\mbox{span}\left(\{a_{s,t}\colon s,t\in G\}\right)$ is dense in $\ell^1(G,C_0(G),\alpha)$, and hence also in $F^p(G,C_0(G),\alpha)$.

Let $S$ be a finite subset of $G$. Then $\{a_{s,t}\colon s,t\in S\}$ is a standard system of matrix units for $M_{|S|}$, so the subalgebra $M_S$ of $F^p(G,C_0(G),\alpha)$ they generate is canonically
isomorphic, as a Banach algebra, to $M_{|S|}^p$. We claim that this isomorphism is isometric, this is, that the norm that $M_{S}$ inherits as a subalgebra of $F^p(G,C_0(G),\alpha)$ is the standard norm
of $M_{|S|}^p$. To check this, it will be enough to show that if
$$\rho\colon \ell^1(G,C_0(G),\alpha) \to \B(L^p(X,\mu))$$
is a nondegenerate contractive representation on a $\sigma$-finite
measure space $(X,\mu)$, then the restriction
$$\rho|_{M_{S}}\colon M_{S}\to \B(L^p(X,\mu))$$
is spatial in the sense of Definition 7.1 in \cite{Phi12arX:LpCuntz}.

Given such a representation $\rho\colon \ell^1(G,C_0(G),\alpha) \to \B(L^p(X,\mu))$,
let $\pi\colon C_0(G)\to \B(L^p(X,\mu))$ be the nondegenerate contractive representation, and
let $v\colon G\to \Isom(L^p(X,\mu))$ be the isometric group representation such that
$(\pi,v)$ is the covariant representation of $(G, C_0(G),\alpha)$ whose integrated form is $\rho$.
For $s$ and $t$ in $S$, one has
$$\rho(a_{s,t})=\pi(\chi_{\{s\}})v_{st^{-1}}.$$

Since $\chi_{\{s\}}$ is a hermitian idempotent in $C_0(G)$, it follows that $\pi(\chi_{\{s\}})$ is also hermitian.
Use Example 1.1 in \cite{Ber72HermitianProj} to choose
a measurable subset $F$ of $X$ such that $\pi(\chi_{\{s\}})=\chi_F$. Since $v_{st^{-1}}$ is a bijective isometry, it is spatial. If $(X,X,T,g)$ is a spatial system for
$v_{st^{-1}}$ (see Definition 6.1 in \cite{Phi12arX:LpCuntz}), then it is easy to check that $(F,X,T,g)$ is a spatial system for $\rho(a_{s,t})$. This shows that $\rho(a_{s,t})$ is a spatial partial isometry, and hence $\rho|_{M_S}$ is a spatial representation.

Denote by $\mathcal{F}$ the upward directed family of all finite subsets $S$ of $G$. For each $S$ in $\mathcal{F}$, let $\varphi_S\colon M_{|S|}^p\to F^p(G,C_0(G),\alpha)$
be the canonical isometric isomorphism that sends the standard matrix units of $M_{|S|}$ to the set of matrix units $\{a_{s,t}\colon s,t\in S\}$. It is clear that there is an isometric homomorphism
$$\varphi_0\colon \bigcup_{S \in \mathcal{F}} M_{|S|}^p\to F^p(G,C_0(G),\alpha)$$
whose range contains $\{a_{s,t}\colon s,t\in G\}$. Note that $\bigcup\limits_{S \in \mathcal{F}} M_{|S|}^p$ is a subalgebra of $\K(\ell^p(G))$. Since $\varphi_0$ is isometric, it extends by continuity to an isometric
homomorphism
$$\varphi\colon \overline{\bigcup_{S \in \mathcal{F}} M_{|S|}^p}\to F^p(G,C_0(G),\alpha),$$
which must be surjective since its range is dense. This is the desired isometric isomorphism.

We will now show that the canonical map
\[\kappa\colon F^p(G, C_0(G),\alpha)\to F^p_\lambda(G, C_0(G),\alpha)\]
is an isometric
isomorphism.
The usual argument for $C^*$-algebras is that $\kappa$ is surjective, and since $F^2(G, C_0(G),\alpha)\cong \K(\ell^2(G))$ is simple, $\kappa$ must be an isomorphism.
However, when $p\neq 2$, we do not know whether $\kappa$ has closed range.
Here is where incompressibility comes into play.

The full crossed product $F^p(G, C_0(G),\alpha)$ is $p$-incompressible by \autoref{lem: DirLimIncomp}, because
it is the direct limit of the $p$-incompressible Banach algebras $M_{S}^p$ (see Theorem 7.2 in \cite{Phi12arX:LpCuntz}).
Since $F^p_\lambda(G,C_0(G),\alpha)$ can be isometrically represented on an $L^p$-space, it follows that $\kappa$ must
be isometric. Finally, having dense range, $\kappa$ is an isometric isomorphism. This finishes the proof.
\end{proof}

\providecommand{\bysame}{\leavevmode\hbox to3em{\hrulefill}\thinspace}
\providecommand{\noopsort}[1]{}
\providecommand{\mr}[1]{\href{http://www.ams.org/mathscinet-getitem?mr=#1}{MR~#1}}
\providecommand{\zbl}[1]{\href{http://www.zentralblatt-math.org/zmath/en/search/?q=an:#1}{Zbl~#1}}
\providecommand{\jfm}[1]{\href{http://www.emis.de/cgi-bin/JFM-item?#1}{JFM~#1}}
\providecommand{\arxiv}[1]{\href{http://www.arxiv.org/abs/#1}{arXiv~#1}}
\providecommand{\doi}[1]{\url{http://dx.doi.org/#1}}
\providecommand{\MR}{\relax\ifhmode\unskip\space\fi MR }
\providecommand{\MRhref}[2]{%
  \href{http://www.ams.org/mathscinet-getitem?mr=#1}{#2}
}
\providecommand{\href}[2]{#2}

\end{document}